\documentclass[reqno, 12pt]{amsart}
\usepackage{amscd}
\usepackage[all]{xy}
\usepackage{graphicx}
\usepackage{amsmath}
\usepackage{color}

\usepackage{amsmath, latexsym, amssymb}
\usepackage{epstopdf}
\numberwithin{equation}{section}
\theoremstyle{plain}
\newtheorem{lemma}{Lemma}[section]
\newtheorem{proposition}[lemma]{Proposition}
\newtheorem{theorem}[lemma]{Theorem}
\newtheorem{corollary}[lemma]{Corollary}

\theoremstyle{definition}
\newtheorem{definition}[lemma]{Definition}
\newtheorem{remark}[lemma]{Remark}

\begin{document}
\newcommand{\R}{{\mathbb R}}
\newcommand{\C}{{\mathbb C}}
\newcommand{\D}{{\mathbb D}}
\newcommand{\F}{{\mathbb F}}
\renewcommand{\O}{{\mathbb O}}
\newcommand{\Z}{{\mathbb Z}} 
\newcommand{\N}{{\mathbb N}}
\newcommand{\Q}{{\mathbb Q}}
\renewcommand{\H}{{\mathbb H}}

\newcommand{\Aa}{{\mathcal A}}
\newcommand{\Bb}{{\mathcal B}}
\newcommand{\Cc}{{\mathcal C}}    %configuration space
\newcommand{\Dd}{{\mathcal D}}
\newcommand{\Ee}{{\mathcal E}}
\newcommand{\Ff}{{\mathcal F}}
\newcommand{\Gg}{{\mathcal G}}    %gauge transformations
\newcommand{\Hh}{{\mathcal H}}
\newcommand{\Kk}{{\mathcal K}}
\newcommand{\Ii}{{\mathcal I}}
\newcommand{\Jj}{{\mathcal J}}
\newcommand{\Ll}{{\mathcal L}}
\newcommand{\Llt}{{\tilde{\mathcal L}}}    
\newcommand{\Mm}{{\mathcal M}}    
\newcommand{\Nn}{{\mathcal N}}
\newcommand{\Oo}{{\mathcal O}}
\newcommand{\Pp}{{\mathcal P}}
\newcommand{\Qq}{{\mathcal Q}}
\newcommand{\Rr}{{\mathcal R}}
\newcommand{\Ss}{{\mathcal S}}
\newcommand{\Tt}{{\mathcal T}}
\newcommand{\Uu}{{\mathcal U}}
\newcommand{\Vv}{{\mathcal V}}
\newcommand{\Ww}{{\mathcal W}}
\newcommand{\Xx}{{\mathcal X}}
\newcommand{\Yy}{{\mathcal Y}}
\newcommand{\Zz}{{\mathcal Z}}

\newcommand{\zt}{{\tilde z}}
\newcommand{\xt}{{\tilde x}}
\newcommand{\Ht}{\widetilde{H}}
\newcommand{\ut}{{\tilde u}}
\newcommand{\Mt}{{\widetilde M}}

\newcommand{\yt}{{\tilde y}}
\newcommand{\vt}{{\tilde v}}
\newcommand{\wt}{{\tilde w}}

\newcommand{\Ppt}{{\widetilde{\mathcal P}}}
\newcommand{\bp }{{\bar \partial}} 

\newcommand{\Remark}{{\it Remark}}
\newcommand{\Proof}{{\it Proof}}
\newcommand{\ad}{{\rm ad}}
\newcommand{\Om}{{\Omega}}
\newcommand{\om}{{\omega}}
\newcommand{\eps}{{\varepsilon}}
\newcommand{\Di}{{\rm Diff}}
\newcommand{\im}{{\rm Im}\,}
\newcommand{\coker}{{\rm coker }\, }
\newcommand {\sppt}{{\rm sppt}\,}
\newcommand{\rk}{{\rm rank}\,}
\newcommand{\Pro}[1]{\noindent {\bf Proposition #1}}
\newcommand{\Thm}[1]{\noindent {\bf Theorem #1}}
\newcommand{\Lem}[1]{\noindent {\bf Lemma #1 }}
\newcommand{\An}[1]{\noindent {\bf Anmerkung #1}}
\newcommand{\Kor}[1]{\noindent {\bf Korollar #1}}
\newcommand{\Satz}[1]{\noindent {\bf Satz #1}}

\renewcommand{\a}{{\mathfrak a}}
\renewcommand{\b}{{\mathfrak b}}
\newcommand{\e}{{\mathfrak e}}
\renewcommand{\k}{{\mathfrak k}}
\newcommand{\pg}{{\mathfrak p}}
\newcommand{\g}{{\mathfrak g}}
\newcommand{\gl}{{\mathfrak gl}}
\newcommand{\h}{{\mathfrak h}}
\renewcommand{\l}{{\mathfrak l}}
\newcommand{\sm}{{\mathfrak m}}
\newcommand{\n}{{\mathfrak n}}
\newcommand{\s}{{\mathfrak s}}
\renewcommand{\o}{{\mathfrak o}}
\newcommand{\so}{{\mathfrak so}}
\renewcommand{\u}{{\mathfrak u}}
\newcommand{\su}{{\mathfrak su}}
\newcommand{\ssl}{{\mathfrak sl}}
\newcommand{\ssp}{{\mathfrak sp}}
\renewcommand{\t}{{\mathfrak t }}
\newcommand{\Cinf}{C^{\infty}}
\newcommand{\la}{\langle}
\newcommand{\ra}{\rangle}
\newcommand{\half}{\scriptstyle\frac{1}{2}}
\newcommand{\p}{{\partial}}
\newcommand{\notsub}{\not\subset}
\newcommand{\iI}{{I}}               %unit interval [0,1]
\newcommand{\bI}{{\partial I}}      %boundary of same
\newcommand{\LRA}{\Longrightarrow}
\newcommand{\LLA}{\Longleftarrow}
\newcommand{\lra}{\longrightarrow}
\newcommand{\LLR}{\Longleftrightarrow}
\newcommand{\lla}{\longleftarrow}
\newcommand{\INTO}{\hookrightarrow}

\newcommand{\QED}{\hfill$\Box$\medskip}
\newcommand{\UuU}{\Upsilon _{\delta}(H_0) \times \Uu _{\delta} (J_0)}
\newcommand{\bm}{\boldmath}

\title[Periodic solutions  of locally Hamiltonian equations]{\large The Calabi invariant and  the  least number of  periodic solutions of locally Hamiltonian equations
}

\author{H\^ong V\^an L\^e}
\address{Institute of Mathematics of ASCR, Zitna 25, 11567 Praha 1, Czech Republic.}
\email{hvle@math.cas.cz}

\date{\today}

\thanks{H. V. L. is  supported by RVO: 67985840}
\abstract  In this paper   we    prove     a  lower  bound  for the least number   of  one-periodic solutions  of  nondegenerate locally
Hamiltonian  equations  on compact  symplectic  manifolds  in terms of  the   Betti numbers of the Novikov homology associated  to the  Calabi invariant  of the  locally Hamiltonian equations. Our  result  improves     lower bounds  obtained by  L\^e-Ono  and  Ono   for the least number 
of   nondegenerate   locally Hamiltonian     symplectic fixed points.   Our result also generalizes    the homological Arnold conjecture  that  has been proved by     Fukaya-Ono and Liu-Tian.

\endabstract

\maketitle

\tableofcontents

\section{Introduction}

Periodic  solutions, after  stationary   points,  are simplest  objects in the  qualitative theory  of  dynamical systems.
In this   paper  we study one-periodic solutions  of   locally  Hamiltonian systems   on compact symplectic manifolds $(M^{2n}, \om)$.
Given a symplectic form $\om$ on $M^{2n}$,  there is an  isomorphism $L_\om: TM^{2n} \to    T^*M^{2n}$  that satisfies  the following equation:
$$\la L_\om (V),W \ra:= - \om (V,W)$$
%%%%%%%%%\footenote{this  agrees  with \cite[\S 5.4. p. 134]{AM2014}} %%%%%%%%%%
for all $V, W \in   TM^{2n}$.

Recall that a vector field  $V$  on $M^{2n}$ is said to be {\it locally Hamiltonian}, if $L_\om (V)$ is a closed 1-form on  $M^{2n}$ \cite{AM1987}.
A dynamical  system on $M^{2n}$
\begin{equation}
\frac{d}{dt} x(t) = V_t(x(t)) \label{eq:lham1}
\end{equation}
 is called {\it locally Hamiltonian}, if  $V_t$ is  a    locally Hamiltonian  vector field  on $M^{2n}$ for all $t$. We refer the reader to \cite[Chapter 16]{Tarasov2008} for   classical examples of locally Hamiltonian systems  and  to \cite{Farber2004}, \cite{FJ2003} for topological  consideration
 of    flows     generated by time independent  locally Hamiltonian  vector fields.
 
 Let $\varphi_t: M^{2n} \to M^{2n}$ be the flow
generated by the locally Hamiltonian vector fields  $V_t$  in (\ref{eq:lham1}).   
 Clearly, the set  of the fixed   points   of the time-one map $\varphi_1$  is in 1-1  correspondence   with the set 
 of  one-periodic  solutions of    (\ref{eq:lham1}), i.e. those solutions $x(t)$   with  $x(0) = x(1)$.
If we are interested only in one-periodic  solutions of    (\ref{eq:lham1}) we  can assume  w.l.o.g.  that $V_t$ is one-periodic in $t$, i.e.
$V_t = V_{t+1}$ for all $t$ \cite{LO1995}.

An important invariant  of the time-one map $\varphi_1$ is its  Calabi invariant, defined as  follows \cite{Banyaga1978}:
$$Cal(\varphi_1): = [\int_0^1 L_\om  (V_t) \, dt]\in  H^1 (M^{2n}, \R).$$ 
In \cite[Deformation Lemma 2.1]{LO1995}  L\^e-Ono showed that
there  exist  a one-periodic   Hamiltonian  function $H \in C^\infty (S^1\times  M^{2n})$   and  a closed 1-form  $\theta \in \Om ^1 (M^{2n})$ such  that the  time-one  map $\varphi_1$  associated with (\ref{eq:lham1}) is the solution  at time  $t=1$  of   the following  equation
\begin{equation}
\frac{d}{dt} \varphi_t (x) = L_\om ^{-1} (\theta +  dH_t) (\varphi_t(x)), \: \varphi_0  = Id \label{eq:lham2}
\end{equation}
where  $H_t (x) : = H(t, x)$,  and  $[\theta]  =  Cal (\varphi_1)$.  
 Henceforth  the  set of  one-periodic solutions  of (\ref{eq:lham1})  coincides  with
 the set  of one-periodic solutions  of  the following    equation
\begin{equation}
\frac{d}{dt} x(t) = L_\om ^{-1} (\theta +  dH_t)(x(t)). \label{eq:lham3}
\end{equation}
Thus, in the   present paper  we  consider only  one-periodic solutions  of  locally Hamiltonian  equations of the form (\ref{eq:lham3}). 
We    shall  also call $[\theta]$ {\it the Calabi invariant of the equation (\ref{eq:lham3})}.

 A  one-periodic  solution of   (\ref{eq:lham3}) is called {\it nondegenerate}, if  the fixed point $x(0)$ of  the associated time-one map  $\varphi_1$   is nondegenerate, or equivalently, 
$\det (Id - d\varphi_1(x(0))) \not=0$.  A  locally Hamiltonian equation  (\ref{eq:lham3})  is called {\it nondegenerate}, if  all    one-periodic solutions
of (\ref{eq:lham3})  are nondegenerate.  Since  nondegenerate fixed points of a diffeomorphism  are isolated,   a nondegenerate locally  Hamiltonian  equation on a compact  symplectic manifold $M^{2n}$ has only a finite number of  one-periodic  solutions.

In \cite{LO1995}  L\^e-Ono  introduced    Floer-Novikov   chain complexes  associated  with  nondegenerate locally Hamiltonian equations  of the  form (\ref{eq:lham3})
on  compact  weakly monotone   symplectic  manifolds $(M^{2n}, \om)$, see  also  section \ref{sec:fn}  below. As a result,  L\^e-Ono obtained the following.

\begin{proposition}\label{prop:leono} (\cite[Main Theorem]{LO1995})  Let $(M,\omega )$ be a closed  symplectic 
manifold of dimension $2n$ which satisfies the following condition
$$c_{1 \ |\pi _2 (M)} = \lambda \om _{|\pi _2 (M)}, \; \lambda \neq 0,$$ 
and if $\lambda <0$, the minimal Chern number $N$ satisfies $N>n-3$.   
Suppose $\varphi_1$ is the time-one  map of   the flow 
associated to (\ref{eq:lham3}).  
If all the fixed points of $\varphi_1$ are nondegenerate, 
then the  number of fixed points of $\varphi_1$ is at least the sum of the 
Betti 
numbers of the Novikov homology over ${\bold Z}_2$ associated to the 
Calabi 
invariant of $\varphi_1$. 
\end{proposition}

The restriction of Proposition \ref{prop:leono}  to  the class of positively or negatively monotone  symplectic manifolds  is caused by the difficulty in  computing
 the Floer-Novikov cohomology  which depends  on  the    Calabi invariant  $[\theta]\in H^1 (M^{2n}, \R)$  of $\varphi_1$.  In \cite{Ono2005}  Ono  refined the energy estimate in \cite{LO1995}  in order  to show  that    Floer-Novikov  chain complexes  can be defined  over  Novikov rings that are smaller than the one defined in \cite{LO1995}.  Using in addition  the construction of the Kuranishi structure  proposed by Fukaya-Ono in \cite{FO1999},  he  
proved another   variant
of   Proposition \ref{prop:leono} as follows.

\begin{proposition}\label{prop:ono} (\cite[Theorem 1.1]{Ono2005})  Let  $(M^{2n}, \om)$  be a  compact  symplectic  manifold. 
Suppose $\varphi_1$ is the time-one  map of   the flow 
associated to (\ref{eq:lham3}).   
If all fixed points of  $\varphi$ are nondegenerate, the number of fixed points  $Fix (\varphi)$ of $\varphi$ is
not less than $\sum _p \min-{\rm nov}^p (M)$.
\end{proposition}

Let us recall   the definition of  $\min-{\rm nov}^p (M)$ introduced by   Ono in \cite{Ono2005}.
For $a \in H^1 (M, \R)$ we denote by  $HN^*(M; a)$   the  Novikov cohomology  over  $\Q$ associated  with $a$.
 The function $a \mapsto  \rk HN^*(M; a)$ attains the absolute minimum
at generic $a$. Denote by $\min-{\rm nov}^p(M)$ the minimum of rank $HN^p(M; a)$,
which we call the $p$-th minimal Novikov number\cite{Ono2005}.  The  number  $\min-{\rm nov}^p(M)$  appears  because
it is also difficult to  controll  the family of  Floer-Novikov chain complexes  as $\theta$  varies.

\

Denote by $\Pp(\om, \theta, H)$ the set of all contractible one-periodic  solutions of (\ref{eq:lham3}).
The goal  of this paper   is  to  prove the following.

\begin{theorem}[Main Theorem]\label{thm:main}  Let  $(M^{2n}, \om)$  be a  compact  symplectic  manifold.        Assume  that  
all the  contractible   one-periodic  solutions of  (\ref{eq:lham3})  are  nondegenerate.
Then  the  cardinal of the set $\Pp(\om, \theta, H)$  is at least the sum of the Betti 
numbers of the Novikov homology over $\Q$ associated to the Calabi invariant of (\ref{eq:lham3}). 
If moreover $(M^{2n}, \om)$  is weakly monotone, then   for any field $\F$ the   cardinal of the set $\Pp(\om, \theta, H)$  is at least the sum of the Betti 
numbers of the Novikov homology over $\F$ associated to the Calabi invariant of (\ref{eq:lham3}). 
\end{theorem}

We obtain   from Theorem \ref{thm:main} immediately the following generalization of Propositions \ref{prop:leono}, \ref{prop:ono}. %, see also Remark \ref{rem:nov} below.

\begin{corollary}\label{cor:main}  Let  $(M^{2n}, \om)$  be a  compact  symplectic  manifold. 
Suppose $\varphi_1$ is the time-one  map of   the flow 
associated to (\ref{eq:lham3}).   
If all fixed points of  $\varphi$ are nondegenerate, the number of fixed points  $Fix (\varphi)$ of $\varphi$ is
at least the sum of the 
Betti 
numbers of the Novikov homology over $\Q$ associated to the 
Calabi 
invariant of $\varphi$. If $(M^{2n}, \om)$ is weakly monotone,  then we can replace the sum of the 
Betti 
numbers of the Novikov homology over $\Q$ by the sum of the 
Betti 
numbers of the Novikov homology over $\F$  for any field  $\F$.
\end{corollary}

\begin{remark}\label{rem:nov}    Corollary \ref{cor:main} also generalizes   different versions  of the homological  version of the Arnold conjecture  for Hamiltonian symplectic fixed points that have  been  proved by  Ono \cite{Ono1995} for  compact weakly monotone  symplectic   manifolds,   and by Fukaya-Ono \cite{FO1999} and Liu-Tian \cite{LT1998}  for general  compact symlectic  manifolds.

%2. If  $\rk H^1 (M^{2n}, \R) = 1$,  then   Corollary \ref{cor:main} is equivalent  to Proposition \ref{prop:ono} due to Ono. If
%$\rk H^1 (M^{2n}, \R) \ge 2$, then  there  exists   $a \in  H^1 (M^{2n}, \R)$ such that  $ \sum _i b_i(HN ^* (M^{2n}; a)) > \sum _p \min -{\rm nov} ^p (M^{2n})$.
%For example, if $M^{2n} =   T^2$,  $a  = ( 1, 0) \in   H^1(S^1, \R) \oplus H^1 (S^1, \R) = H^1 (T^2, \R)$ then  $\sum _p \min -{\rm nov} ^p (T^2) = 0$, and
%$\sum_i b_i (NH^* (M^{2n}; a)) = 2$.

%3. Proposition \ref{prop:leono} due to L\^e-Ono   is also  valid  if we replace the Novikov homology over $\Z_2$ by the Novikov homology over any   field  $\Z_p : =Z/p \Z$, since we can define  the Floer-Novikov  chain complexes with coefficients in $\Z$  on  compact weakly monotone  symplectic
%manifolds. %Taking into account  the announced  results in \cite{FO2001}  we  might    drop     the (weakly monotone) condition $N \ge n-2$ in Proposition \ref{prop:leono}.
\end{remark}

To prove Theorem \ref{thm:main} in the  case  $(M^{2n}, \om)$ is  weakly monotone, we first   show that  the   underlying Novikov ring  $\Lambda ^R_{\theta, \om}$ of   Novikov-Floer chain complexes
with coefficient in  an integral domain $R$ is an integral domain (Proposition \ref{prop:compono}). Thus the Betti numbers of the Floer-Novikov homology  groups  are well-defined. Then we  introduce   notions of   an admissible family of  nondegenerate (multi-valued)  Hamiltonian  functions and its good  neighborhood, which are    generalization  and formalization  of the notion of  special  deformations of a nondegenerate 
symplectic isotopy   that has been introduced  in \cite{LO1995} and  refined in \cite{Ono2005}.
Using a  good neighborhood  of an admissible   family of nondegenerate (multi-valued) Hamiltonian  functions we %estimate from below  the number  of the one-periodic   solutions of (\ref{eq:lham3})  by  the sum of the Betti numbers  of a  Floer-Novikov  chain complex  associated to a  rational     cohomology class  $ [\theta']$  where $ \theta'$ is  close  to $\theta$ (Proposition \ref{prop:perttheta}).
compare  the Betti numbers of two  ``close"  Floer-Novikov   chain complexes, using and extending    results and ideas  in \cite{LO1995,Ono2005}. Then we compute the  Betti numbers   of the homology  of a refined  Floer  chain complex  on the   minimal covering $\Mt^{2n}$ of $M^{2n}$  associated with $[\theta]$, using  standard  arguments in  Floer   theory.    

\

Our paper is organized as follows.  In section  \ref{sec:fn} we recall the construction  of  Floer-Novikov chain complexes  on compact  weakly monotone   symplectic manifolds,  following \cite{LO1995}.
In section  \ref{sec:betti}  we   compute the  Betti  numbers  of  the  Floer-Novikov  homology  with coefficient in  an integral  domain $R$  in the case of  compact weakly monotone    symplectic manifold $(M^{2n}, \om)$. In section \ref{sec:main} we  prove our main  theorem.   Finally in section \ref{sec:concl}  we discuss  some  open problems.

\section{Floer-Novikov  chain complexes  on  compact weakly monotone symplectic  manifolds}\label{sec:fn}
  In this section  we summarize the construction   of  Floer-Novikov  chain complexes  on     compact weakly monotone symplectic manifolds $(M^{2n}, \om)$, following \cite{LO1995}.  
%We always  consider     nondegenerate  locally  Hamiltonian equation  (\ref{eq:lham3})\blue{ (DO  I need  to   add  a proof of the density  of nondegenerate equation  for fixed $[\theta]$?)}
%%%% \footnote{Convention is    as in  \cite{LO1995}. In \cite{McDS2004}  they changed the sign of   whole   action function, so they  have    graident  flow, not minus  gradient flow.} %%%%%%%%

%\subsection{Weakly monotone symplectic  manifolds}
% Let $(M^{2n}, \om)$ be   a compact  weakly   monotone   symplectic manifold, i.e.
%for every $A \in  \pi_2(M^{2n})$
%$$3 - n \le  c_1(A) < 0  \LRA  \om (A) \le 0,$$
%where $c_1 : =c_1(M^{2n}, \om)$ denotes the  first Chern class of the   almost complex structure that is compatible  with $\om$.
%\begin{lemma}\label{lem:monotone}(\cite[Lemma 1.1]{HS1994}) A compact symplectic manifold $(M^{2n}, \om)$ is weakly monotone if and
%only if one of the following conditions is satisfied.\\
%(a) $\om(A) = \lambda \cdot c_1(A)$ for every $A \in \pi_2(M^{2n})$ where  $\lambda \ge 0$ ($(M^{2n},\om)$ is monotone).\\
%(b) $c_1(A) = 0$ for every $A \in \pi_2(M^{2n})$  ($(M^{2n},\om)$ is   spherical  Calabi-Yau).\\
%(c) The minimal Chern number $N \ge 0$ defined by $c_1(\pi_2(M^{2n})) = N\Z$ is greater
%than or equal to $n - 2$.
%\end{lemma}

We always assume in this paper that   the  equation (\ref{eq:lham3}) is nondegenerate. We identify the set $\Pp (\om, \theta, H)$ of  contractible one-periodic solutions of (\ref{eq:lham3}) 
with  the zero-set of the following closed 1-form $d\Aa _{(\theta, H)}$ on the loop space  over $M^{2n}$:
\begin{equation}
d\Aa _{\theta, H} (x, \xi) = \int \om ( \dot{x}, \xi) + (\theta+  dH _t) (x(t))(\xi).\label{eq:ham2}
\end{equation}
%%%%This  agrees  with \cite[p.161]{AD2014}%%%%
%We  shall consider  a covering  of $\LlM$  on which   the pull back  of $d\Aa _{(\theta, H)}$  is an exact   1-form.
%Then the   associated  $L_2$-gradient  of the multi-valued functional $\Aa _{\theta, H}$  is    $-J(u)({\p u\over \p t} - X_{\theta, H, t} (u))$
%and the corresponding   ``gradient flow'' for the multi-valued functional $\Aa _{\theta, H}$
%is defined by the following equation: 
%\begin{equation}
%\overline{\p} _{J,\theta, H} (u):= {\p u\over \p s} + J(u)({\p u\over \p t} - X_{\theta, H, t} (u))=0,\label{eq:ham3}
%\end{equation}
%where $u=u(s,t)$ is a mapping $\R \times S^1 \rightarrow M^{2n}$. If  a solution $u$ of  (\ref{eq:ham3}) satisfies  
%the boundary condition 
%$$\displaystyle{\lim_{s \to \pm \infty} u(s,t)} = x^{\pm}(t)\in \Pp (\om, \theta, H)$$
%we call  $u$ {\it a connecting orbit}.
%For a solution
%$u$ of (\ref{eq:ham3}) we define its energy as follows.
%\begin{equation}
%E(u) ={1\over 2} \int_{-\infty}^{\infty} \int_0^1 (\vert{\p u\over \p s}\vert ^2 +
%\vert{\p u\over \p t} - X_{\theta, H, t}(u)\vert^2)\, dtds. \label{eq:energy}
%\end{equation}
%The  norm in (\ref{eq:energy})  is  defined by  $g_J$.   In what  follows  we  shall omit  subscript $g_J$ if
%there is  no  misunderstanding.
 
We will restrict ourselves to the component  $\Ll M ^{2n}$ of contractible loops on
$M ^{2n}$.  We construct an associated covering space $\Llt  \Mt^{2n}$  of $\Ll M^{2n}$
such that the   pull back of  $d\Aa _{\theta, H}$ on this cover is
an exact 1-form.
Consider the following commutative diagram:
$$\begin{array}{ccccc}
\widetilde{\Ll} \Mt^{2n}& \stackrel{\tilde{j}}{\lra} & \Ll \Mt &
\stackrel{\tilde{e}}{\lra } & \Mt ^{2n} \\
\Big\downarrow\vcenter{\rlap{${\tilde{\Pi}}$}} & & 
\Big\downarrow\vcenter{ \rlap{${\Pi }$}} & &
\Big\downarrow\vcenter{ \rlap{$\pi$ }} \\ 
\widetilde{\Ll} M^{2n} & \stackrel{j}{\lra} & \Ll M^{2n} & \stackrel{e}{\lra } & M^{2n}. 
\end{array}$$

Here $\Mt ^{2n} $ denotes the covering space of $M^{2n}$ associated to the period 
homomorphism of $\theta$, 
$I_{\theta}:\pi_1(M^{2n}) \rightarrow {\bold R}$.  
In other words the covering transformation group of $\Mt$
is isomorphic to the quotient group 
$$\Gamma _1 =H _1(M^{2n},\Z) /\ker I_{\theta}.$$  

In the above diagram $e$ denotes 
the evaluation map $x \mapsto  x(0)$ and    $\Llt M^{2n} $  is the covering  of $\Ll M^{2n}$  whose desk transformation group is
%Let $j$ denotes the projection
%from the covering space $\Llt M^{2n} $ of $\Ll M^{2n}$ associated to the
%action of 
%This means that the covering transformation group is isomorphic to the 
%quotient group 
$$\Gamma _2 = {\pi _2 (M^{2n}) \over \ker I_{c_1} \cap \ker I _{\om}},$$
where the homomorphisms $I _{c_1}$, $I_{\om}$: $\pi _2(M^{2n}) \to \R$
 are defined by evaluating $c_1$ and $[\omega]$ respectively.
Thus, an element of $\widetilde{\Ll}\Mt ^{2n}$ is represented by an equivalence class 
of pairs $(\tilde x,\tilde w)$, where $\tilde x$ is a loop in $\Mt  ^{2n}$,
$\tilde w$ is a disk in $\Mt ^{2n}$ bounding $\tilde x$.  The pair
 $(\tilde x, \tilde w)$ is 
equivalent to $(\tilde y, \tilde v)$ if and only if $\tilde x=\tilde y$ and 
the values of $I_{c_1}$ and $I_{\om}$ are zero on $w \# (-v)$, 
where $w=\pi (\tilde w), v=\pi (\tilde v)$.    The    covering transformation group $\Gamma$ acts as follows
\begin{equation}
 (\gamma_1 \oplus \gamma_2) [\xt ,\tilde  w ] = [\gamma_1 \cdot \xt , A_2 \# 
\gamma_1\cdot \tilde w ],
\label{eq:aum}
\end{equation}
where $A_2$ is any representative of $\gamma_2$ in $\pi _2 (M^{2n})$.   %Note that for $\gamma ,\gamma ' \in \Gamma$  we have 
%$$(\gamma \circ \gamma ' )[\xt , \tilde w] = (\gamma' \circ \gamma) [\xt , \tilde w ].$$ 
 %Summarizing,  we have the following  important  fact.
Since  the  torsion  part of $\pi_2(M^{2n})$ lies in the intersection $\ker I_{c_1} \cap \ker I_\om$, we obtain  the following
 
\begin{lemma}\label{lem:sum} (cf. \cite[Lemma 2.2]{LO1995})  The covering transformation group  $\Gamma$ of $\widetilde{\Ll}\Mt ^{2n}\to \Ll M^{2n}$ is 
the direct sum  of the   finitely generated  torsion free  abelian groups $\Gamma _1$ and $ \Gamma _2$. 
\end{lemma}

Observe  that there exists a unique  up to a constant Hamiltonian $\Ht \in C^\infty (S^1 \times \Mt ^{2n}) $   
such that 

\begin{equation}
d\Ht_t  =  \pi^*(\theta  + dH_t) \label{eq:lift}
\end{equation}
  for all $t \in S^1$. For the sake of simplicity, we also denote by $\om$ the  symplectic form $\pi^*(\om)$ on $\Mt ^{2n}$. Clearly, the time-dependent Hamiltonian
flow on $\Mt ^{2n} $ generated by $\Ht$ is the pull-back of the original
symplectic flow on $M ^{2n}$.  In particular, the set of contractible
one-periodic solutions 
$$\Pp (\Ht ): = \Pp(\om, 0, \Ht)$$ 
coincides with  the set $\pi ^{-1}( (\Pp (\om, \theta, H )))$. 
Furthermore, $\Ppt(\Ht ):=\tilde{j}^{-1}(\Pp (\Ht))$ is the critical 
set of the
following  action functional  
$$\Aa _{\Ht} ([\xt, \tilde w])=  -\int_D \tilde w ^*\om +\int _0^1 \Ht (t, \xt (t))\, dt. $$

%%%\footnote{agrees with \cite[p.160]{AD2014}}%%%%%%%%%
Denote by $\Jj (M^{2n}, \om)$ the set of all smooth compatible   almost  complex structures on $(M^{2n}, \om)$.  
Let   $\Jj_{reg} (M^{2n}, \om)\subset \Jj(M^{2n}, \om)$ be the subset  of  regular   compatible almost   complex structures,  see \cite{HS1994}  and Remark \ref{rem:comment1} for a short  explanation.
%$$X_{\theta, H, t} : = L_{\om} ^{-1} (\theta + dH_t).$$
For  $J \in  \Jj_{reg} (M^{2n}, \om)$  we also denote by $J$  the  lifted  almost complex structure on $\Mt^{2n}$. 
%For $J \in  \Jj_{reg} (M^{2n}, \om)$
Let us denote by  $g $  the associated  Riemannian metric on $\Mt$.  % 
%Set 
Using  $\om  (X, Y) = g (JX, Y)$ we obtain
$$X_{\Ht_t} : = L_{\om} ^{-1}(d\Ht_t) = J\nabla  \Ht_t$$
where $\nabla$ denotes the   gradient w.r.t the Riemannian metric $g$. %The Riemannian metric $g$   induces  the associated  $L^2$-metric  on the  loop spaces $\Ll M^{2n}$
%and $\Ll \Mt^{2n}$. (This agrees  with \cite[p.144]{AD2014})
We now consider the space  $\Mm ([\xt ^-, \tilde w^-], [\xt ^+, \tilde w ^{+}]; \Ht, J)$ of {\it connecting
orbits} $\ut :\R \times S^1 \to \Mt ^{2n}$  satisfying the   equation  of   $L_2$-gradient  flow  on $\Ll\Mt^{2n}$:
\begin{equation}
\overline{\p} _{J,\Ht } (\ut )= {\p \ut\over \p s} + J(u)({\p \ut \over \p t} - X_{\Ht _t} (\ut))=0,\label{eq:conl}
\end{equation}
%%% agrees  with \cite[p.163]{AD2014}%%%
with the following boundary conditions 
\begin{eqnarray}
\lim _{s \to \pm \infty} \ut (s,t) = \xt ^{\pm} (t)  \in \Pp(\Ht)\label{eq:conb}\\
\, [\tilde x^-, \tilde w^- \# \tilde u]=[\tilde x^+,\tilde w^+].\label{eq:conh}
\end{eqnarray}
%%%%\footenote{consistent  with the flow equation   (12.1.3) in \cite[p.454]{McDS2004}}%%%%%%% 
The following energy identity for  
 $u \in \Mm ([\xt ^-, \ut^-], [\xt ^+, \ut ^{+}]; \Ht, J)$  is crucial in  the theory of Floer(-Novikov)  chain complexes:
\begin{equation}
E(\ut) = \int _{-\infty}^{\infty} \int_0^1 \vert{\p \ut\over \p s}\vert^2\, dtds = \Aa _{\Ht} ([\xt ^- , \tilde w^-]) - \Aa _{\Ht}([\xt ^+, \tilde w ^+]). \label{eq:energyid}
\end{equation}
The dimension  of the  space  of   connecting orbits  is  computed as follows 
$$\dim \Mm ([\xt ^-, \ut^-], [\xt ^+, \ut ^{+}]; \Ht, J)= \mu ([\xt ^-, \tilde w ^-])
-\mu ([\xt ^+ , \tilde w^+]), $$
where  $\mu ([\xt, \tilde w])$ is the Conley-Zehnder index   of  $[\xt , \tilde w ]$.
%%% agrees with \cite[P. 287]{AD204}%%%
The Conley-Zehnder index 
 satisfies the following identity:
$$ \mu ([ \xt , A\# \tilde w ]) -  \mu ([\xt, \tilde w]) = 2c_1(A)  
\text{~for~} A \in \pi_2(M^{2n}).$$

Let $N$ be the minimal Chern number of $(M^{2n}, \om)$  and $\xt \in \Ppt  (\Ht)$.
We will write $\mu (\xt) = k \in \Z_{2N} : = \Z/ 2N $ if there is a bounding disk 
$\tilde w$ such that $\mu ([\xt ,\tilde w]) =k \mod  2N$.

For $k \in \Z_{2N}$ we set
 $$\Ppt _k(\Ht ): =\{ [\xt , \ut ]\in \Ppt (\Ht)|\, \mu ([\xt ,\ut ])=k\}.$$
Let $R$ be an integral domain. We define the {\it Floer-Novikov chain groups  $CFN_*( \Ht,  R)$} as follows.
 % whose $k$-th chain group $C_k (\Ht )$ consists of 
\begin{eqnarray*}
CFN_k( \Ht,  R)& : = &\{ \sum \xi_{[\xt ,\tilde w ]} \cdot [\xt ,\tilde w ], \,  [\xt ,\tilde w ] \in \Ppt _k(\Ht ), \; \xi_{[\xt ,\tilde w ]} \in R| \\
& &\text{ for all } c \in {\bold R} \text{ there is    only finite  number of} \\
 & &  [\xt ,\tilde w ]\text{ such  that }  \xi_{[\xt ,\tilde w ]} \neq 0 \, \&\, 
 \Aa_{\Ht}([\xt ,\tilde w ]) > c \}
\end{eqnarray*}

Set
$$\Gamma_2^\circ : ={\ker I _{c_1} \over \ker I _{c_1} \cap \ker I_{\om}}\subset \Gamma_2,$$
$$\Gamma ^\circ : = \Gamma_1 \oplus  \Gamma_2 ^\circ\subset \Gamma.$$
Then  $\Gamma ^0$ is  a finitely generated torsion free  abelian group.
%Then $\phi_\om (\Gamma_0)$  is  a   subgroup of  the abelian  group $\R$,  which is  isomorphic  to $\Gamma_0$.   %Now we  introduce some notations.

%\begin{itemize}
%\item $R((\Gamma_0 ^{\om}))$ -  the (upward)  Novikov completion  of the group ring  $R[[\om](\Gamma_0)]$,  i.e. 
%$\bullet$ $R((\Gamma_0^{c_1}))$ - the Novikov completion  of the  the group ring  $R[c_1(\Gamma_0)]$,\\
%\item $\phi_\om(\Gamma_0) \dot + I_\theta(\Gamma_1)$ - the subgroup  in $\R$ that is  generated by $\phi_\om(\Gamma_0)$  and  $I_\theta(\Gamma_1)$,
%\item $R((\phi_\om(\Gamma_0) \dot + I_\theta(\Gamma_1)))$ - the (upward)  Novikov completion  of the group ring  $R[\phi_\om(\Gamma_0) \dot + I_\theta(\Gamma_1)]$.
%\end{itemize}

%\begin{remark}\label{rem:complofsum}  In  \cite{LO1995}  we  denote the  Novikov  completion $R(([\om](\Gamma_0) \dot + [\theta](\Gamma_1)))$   by
%$\Lambda _{\theta, \om}$, which   we called the   Novikov completion  of  the   group ring  of  $\Gamma_1 \oplus \Gamma _0$  w.r.t. the weight  homomorphism
%$I_\theta \oplus - \phi_{\om}$.  In what  follows we   shall  use both notations  interchangeably. NOW    IT IS O.K.
%\end{remark}

Following \cite{LO1995}  we  denote by $\Lambda^R _{\theta, \om}$ {\it the upward  completion  of  the   group ring }   $R[\Gamma^\circ]$  w.r.t. the weight  homomorphism
$\Psi_{\theta, \om} : = I_\theta \oplus- I_{\om}$, which we also call {\it  the Novikov  ring}.
More precisely, 
\begin{eqnarray}
\Lambda^R _{\theta, \om}& : = &\{ \sum \lambda _g\cdot g, \, g \in \Gamma^\circ , \; \lambda _g\in R\, | \text{ for all } c \in \R \text{ there is    only}\nonumber \\
& & \text{ finite  number of }   g  \text{ such  that } \lambda _g \neq 0 \, \&\, 
 \Psi_{\theta, \om}(g) < c\} \label{eq:novikov}
\end{eqnarray}
%for each $\lambda = \sum  \lambda_g \cdot  g \in \Lambda^R_{\theta, \om} $,  
% where $\lambda_g  \in R$, $g \in \Gamma_1\oplus \Gamma_0$,  and for  every  $C \in  \R$ the number of the terms $\lambda_g\cdot g$ with
%$\Psi_{\theta, \om}(g) < C$ is finite.

%\begin{definition} \label{def:compatible} A triple  $[\theta, \om, c_1]\in  H^1 (M, \Q) \times  H^2 (M, \Q)  \times  H^2 (M, \Z)$  is called  {\it compatible over  a  ring $R$},  if $R[\phi_\om(\Gamma_0)] \subsetR[ I_\theta(\Gamma_1)]$.
%\end{definition}

%\begin{lemma}\label{lem:compatible} If $[\theta, \om, c_1]$  is compatible   over  a ring $R$ then for any $a \in \Q ^+$ we also have  $[a \cdot \theta, \om, c_1]$ is     compatible  over a ring $R$.
%Moreover $\Lambda _{\theta, \om}^R  = \Lambda _{a\theta, \om}^R$.
%\end{lemma}
%\begin{proof}
%\end{proof}

The Novikov  ring $\Lambda ^R _{\theta, \om}$
is a commutative ring with unit. It  acts on $CFN_* ( \Ht ,  R)$ in the following way. For   $\lambda = \sum \lambda_g \cdot g \in \Lambda ^R _{\theta, \om}$  and for  $\xi = \sum\xi_{[\xt ,\tilde w ]} \cdot [\xt ,\tilde w ] \in CFN_*(\Ht,  R)$  we let
$$(\lambda * \xi) : = \sum (\lambda * \xi)_{[\xt , \tilde w]}[\xt, \tilde w]$$
where
$$ (\lambda * \xi)_{[\xt , \tilde w]}: = \sum_{g\in \Gamma^0} \lambda _g \xi_{-g\circ [\xt ,\tilde w]}.$$
 %We easily deduce the following lemma.

%\vspace{0.2in} \noindent 
\begin{lemma}\label{lem:chain}(\cite [Lemma 4.2]{LO1995}) For any $k \in \Z_{2N}$, the chain group $CFN_k ( \Ht,  R )$ is a  finitely generated free module 
over the commutative ring $\Lambda _{\theta , \om} ^R$. The rank of this module is the 
cardinal of the set $\Pp _k(\om, \theta, H)$.
\end{lemma}

\begin{remark}\label{rem:changeh}  Lemma  \ref{lem:chain} reflects the fact that the   ground ring $\Lambda_{\theta, \om}^R$ for  the  chain  group $CFN_* ( \Ht, R)$  depends only on the 
cohomology  class $[\theta]$  such that   $\pi^*(\theta) = d\Ht$.  We  now express  this   fact in a slightly different way.
  Using the compactness  of $M$  and the finiteness of $\Pp(\om, \theta, H)$, we   characterize  the chain  group $CFN_* (\Ht, R)$  as follows.  Let  $\Ht'  =  \Ht +  \pi  ^*(df)$ where  $f \in C^\infty  (S^1 \times M^{2n})$.  Then  it is not hard  to  see
\begin{eqnarray*}
CFN_k( \Ht,  R)& : = &\{ \sum \xi_{[\xt ,\tilde w ]} \cdot [\xt ,\tilde w ], \,  [\xt ,\tilde w ] \in \Ppt _k(\Ht ), \; \xi_{[\xt ,\tilde w ]} \in R| \\
& &\text{ for all } c \in {\bold R} \text{ there is    only finite  number of} \\
 & &  [\xt ,\tilde w ]\text{ such  that }  \xi_{[\xt ,\tilde w ]} \neq 0 \, \&\, 
 \Aa_{\Ht'}([\xt ,\tilde w ]) > c \}.
\end{eqnarray*}
%In particular, if   $\Pp(\Mt,  \Ht) = \Pp(\Mt, \Ht')$  then $CFN_k(\Mt, \om, \Ht, J, R)= CFN_k(\Mt, \om, \Ht, J, R)$.
\end{remark}

We regard $R$  as  a  right $\Z$-module,   and  we denote by $1_R$  the  unit  of  $R$.   For a generator $[\xt , \tilde w ]$ in $CFN_k (\Ht, R)$, we define the boundary operator 
$\partial_{(J, \Ht)} $ as
follows.  

\begin{equation}
 \partial_{(J,\Ht)} ([\xt , \tilde w ]): = \sum_{\mu ([\yt , \vt ])=k-1}  1_R\cdot  n([\xt ,\tilde w],
[\yt , \vt])[\yt , \vt] , \nonumber
\end{equation}
where $ n([\xt ,\tilde w],[\yt , \vt])$ denotes  of the 
algebraic number of the solutions  in the space $\Mm ([\xt ,\tilde w],[\yt , \vt]; \Ht , J)/\R$, where $\R$ acts by translation in variable $s$
It is known that
\begin{equation}
\partial_{(J,\Ht)} ([\xt , \tilde w ]) \in CFN_{k-1} (\Ht, R). \label{eq:bo}
\end{equation}

Observe  that $\p _{(J, \Ht)}$ is invariant under the action of $ \Gamma^0$. Taking into  account (\ref{eq:bo}), this allows us to  extend $\partial_{(J,\Ht)}$ as
 a $\Lambda^R _{\theta, \om}$-linear map from $CFN_k (\Ht,R )$ to $CFN_{k-1}(\Ht, R )$.  
Using the standard
 gluing argument 
and the weak compactness argument, see e.g. \cite{AD2014, McDS2004, Schwarz1994}, we 
deduce that $\partial_{(J, \Ht)} ^2=0$.  
The   chain complex $ (CFN_*(\Ht, R) , \p _{J, \Ht})$ is called {\it  the Floer-Novikov chain  complex} associated with  $(\Ht, J)$.
For $k \in \Z_{2N}$, the homology group
$$HFN_k (\Ht, J, R) = {\ker \partial_{J, \Ht} \cap CFN_k(\Ht, R)\over \text{im } 
\partial_{J, \Ht}(CFN_{k+1} (\Ht, R))}$$
is called the  $k$th Floer-Novikov homology group of the  pair $(\Ht, J)$
with coefficients in $\Lambda ^R_{\theta, \om}$.  
%Obviously, they are  $\Z_{2N}$-graded $\Lambda^R_{\theta ,\omega}$-modules.
The following theorem shows that the Floer homology groups are invariant
under exact deformations.  

\begin{proposition}\label{prop:iso} (cf.\cite[Theorem 4.3]{LO1995})  For {\it generic} pairs $(\Ht ^{\alpha}, J^{\alpha})$,
$(\Ht ^{\beta}, J^{\beta})$ such that $\Ht ^{\alpha}- \Ht ^{\beta}= \pi^*(H^{\alpha, \beta})$ for some  $H^{\alpha, \beta}\in C^\infty (S^1 \times  M^{2n})$
there exists a natural  
chain homotopy equivalence
$$\Phi^{\beta , \alpha}:( CFN_* (\Ht ^{\alpha}, R), \p_{(J^{\alpha}, \Ht ^\alpha)})  \rightarrow 
(CFN_*(\Ht^{\beta}, R), \p_{(J^{\beta},\Ht^{\beta})}).$$
\end{proposition}

\begin{remark}\label{rem:comment1}   In our simplified  exposition of  the  theory of  Floer-Novikov   homology  we  did not specify the   regularity condition posed on  a compatible almost complex structure 
$J \in \Jj_{reg} (M^{2n}, \om)$  and we also omit  a $J$-regularity condition on  a nondegenerate  Hamiltonian $\Ht$. These conditions   have been introduced in \cite{HS1994}  for compact weakly monotone  symplectic manifolds $(M^{2n}, \om)$ and extended in \cite{LO1995} for  regular coverings of $(M^{2n}, \om)$.
Roughly speaking,   a  compatible almost  complex  complex structure $J$  is  called {\it regular},  if   the moduli space of $J$-holomorphic  spheres realizing  a homology class $A \in H_2 (M^{2n}, \Z)$ is a manifold  for any $A$.  Given a regular   compatible  almost  complex structure $J$,    a nondegenerate  Hamiltonian $\Ht\in C^\infty (S^1\times \Mt)$  is called   
{\it  $J$-regular},  if  the following three  conditions  hold.\\
(1) The set of points  in $M^{2n}$  that lie on  contractible orbits  in $\Pp (\Ht)$   does not  intersect  with the  set $M_1(J)$ consisting of  points lying  on $J$-holomorphic spheres  of Chern index  at most 1.\\
(2)  The space  of  connecting  orbits  defined by (\ref{eq:conl}), (\ref{eq:conb}) , (\ref{eq:conh}) is   a  finite dimensional manifold.\\
(3)  The set of points  in $M^{2n}$  that lie on  the  connecting orbits  of relative Conly-Zehnder index at most 2    does not  intersect  with the  set $M_0(J)$ consisting of  points lying  on $J$-holomorphic spheres  of Chern index  at most 0.
%Finally it is important to notice  that   given a  regular compatible  almost complex structure $J$    and  a nondegenerate  Hamiltonian  $\Ht\in C^\infty (S^1 \times  \Mt^{2n})$ there
%exists
%the  set $\Hh_{reg}(J)$ of $J$-regular
%Hamiltonians is  dense in the  space  of smooth   functions  $\Ht\in C^\infty (S^1 \times  \Mt^{2n})$  whose  differential $d\Ht$ is  the lift of
%a closed  1-form on $M^{2n}$. %with respect to the topology of uniform
%convergence with all derivatives

%We also  recommend the reader     to \cite{AD2014}    for    detailed  proof of the  construction of the Floer homology in absence of   pseudo-holomorphic   spheres.
\end{remark}

By Proposition \ref{prop:iso} the  Floer Novikov homology  group $HFN_* (\Ht, J, R)$ depends only on  $\Ht, R$. So we shall abbreviate it
as $HFN_*(\Ht, R)$. % or even as $HFN_*(\Ht)$  if no misundertanding   occurs.
We also   abbreviate $(CFN_*(\Ht, R), \p _{( J, \Ht)})$  as  $CFN_*(\Ht,  J, R)$ if   it does not cause  a confusion.

%%%%%%%%%%%%%%%%%%%%%%%%%%%%%%%%%%%%%%%%%%%%%%%%%%%%%%%%%%%%%%%%%%%%%%%%%%%%%%%%%%%%%%%%%%%%%%%%%%%%%%%%%%%%%%%%%%%%%%%%%%%%%%%%%%%%%%%%%%%%%%%%%%%
%%%%%%%%%%%%%%%%%%%%%%%%%%%%%%%%%%%%%%%%%%%%%%%%%%%%%%%%%%%%%%%%%%%%%%%%%%%%%%%%%%%%%%%%%%%%%%%%%%%%%%%%%%%%%%%%%%%%%%%%%%%%%%%%%%%%%%%%%%%%%%%%%%%
%%%%%%%%%%%%%%%%%%%%%%%%%%%%%%%%%%%%%%%%%%%%%%%%%%%%%%%%%%%%%%%%%%%%%%%%%%%%%%%%%%%%%%%%%%%%%%%%%%%%%%%%%%%%%%%%%%%%%%%%%%%%%%%%%%%%%%%%%%%%%%%%%%%%%%

\section{The Betti numbers  of Floer-Novikov  homology}\label{sec:betti}

In this  section we    restrict  ourselves to  the case    of weakly monotone  symplectic  manifolds $(M^{2n}, \om)$ with  minimal Chern number $N$.  
We  fix  a covering $\Mt$ associated with  a    class $[\theta]\in  H^1 (M^{2n}, \R)$. 
First we   show that  the Novikov ring $\Lambda ^R _{\theta, \om}$  is an integral domain  for any  integral domain  $R$ (Proposition \ref{prop:compono}). Hence  the Betti numbers of the Floer-Novikov homology  are well-defined, see  (\ref{eq:ucf1}).
Then we  prove  that  the   Betti numbers  of the Floer-Novikov homology $HFN_*( \Ht, R)$   do not depend  on the choice of $\Ht$ (Theorem \ref{thm:comp}).
For this  purpose   we   first show   that   the chain  complex  $(CFN_*(\Ht), \p_{(J, \Ht)})$   is     an  extension by scalars  of  a    chain complex    with  the same generators  but  defined on   a proper  sub-ring of the  Novikov  ring
(Theorem \ref{thm:small}).  Hence   the   Betti numbers  of the Floer-Novikov homology $HFN_*(\Ht, R)$  are equal to the  Betti numbers   of  the ``smaller''  Floer-Novikov homology groups (Proposition \ref{prop:rank}).   Then we   introduce the notions  of  an admissible   family of nondegenerate (multi-valued) $J$-regular Hamiltonian  functions   and its good neighborhoods  and   study their  properties (Definitions \ref{def:adm}, \ref{def:good}, Theorem \ref{thm:small}, Proposition \ref{prop:rank}). Using the obtained results,  we prove that the Betti numbers  of the Floer-Novikov homology $HFN_*( \Ht, R)$   locally  do not   depend on the ``weight'' of  their Calabi invariant  (Proposition \ref{prop:bettismall}).  Finally we   compute  the  Betti number of the  Floer-Novikov   homology group  $HFN_*(\Ht,R)$, 
where $\Ht$ is a lift of  a    nondegenerate  Hamiltonian on $M^{2n}$, using  the  Piunikhin-Salamon-Schwarz    construction (Theorem \ref{thm:comp2}, Corollary \ref{cor:betti}).

\subsection{Novikov ring $\Lambda ^R_{\theta, \om}$ revisited}\label{subs:nov}
%In this subsection  we   look closer  at the structure  of  $\Lambda ^R_{\om, \theta}$.    
Given a  ring $R$,    a group $\Gamma$ and a homomorphism $\phi: \Gamma \to \R$  we denote by $R (( \Gamma, \phi))$ the upward completion  of  the   group ring    $R[\Gamma]$  w.r.t. the weight  homomorphism
$\phi$.
More precisely,  as in (\ref{eq:novikov}),  we define 
\begin{eqnarray}
R (( \Gamma, \phi))& : = &\{ \sum \lambda _g\cdot g, \, g \in \Gamma , \; \lambda _g\in R\, | \text{ for all } c \in {\bold R} \text{ there is    only}\nonumber \\
& & \text{ finite  number of }   g  \text{ such  that } \lambda _g \neq 0 \, \&\, 
 \phi(g) < c\} \label{eq:novikov1}
\end{eqnarray}
 If $\Gamma$ is a subgroup of $\R$, $e: \Gamma \to \R$ is the natural embedding, then  we  abbreviate $R((\Gamma, e))$  as $R((\Gamma))$.

 In this  paper  we consider  only  commutative rings $R$  with  unit  and without zero divisor, i.e.  $R$  are integral  domains.

\begin{lemma}\label{lem:direct1}  Assume that  $\Gamma$ is a torsion free finitely  generated abelian group and  $\phi: \Gamma \to \R$ is 
a homomorphism. Then
there is a  subgroup  $\Gamma_\phi \subset \Gamma$ such that
$$\Gamma = \ker \phi \oplus \Gamma_\phi.$$
\end{lemma}
\begin{proof} Let  $t_1, \cdots, t_n \in \R$ be  linearly independent generators of  the subgroup $\phi(\Gamma)\subset \R$.  Pick   elements $\gamma_i \in \Gamma$     such  that
$$\phi (\gamma_i) = t_i.$$
Let  $\Gamma_\phi$  be the subgroup in  $\Gamma$  that  is generated  by   $\{\gamma_i|\, i \in [ 1, n]\}$.  Clearly  $\Gamma_\phi \cap  \ker \phi = 0$.
We  shall show that  $\Gamma   = \ker \phi \oplus \Gamma_\phi$.  Let  $\gamma \in  \Gamma$.  If  $\phi (\gamma) \not = 0$ then there are  numbers
$a_i \in \Z$ such that  $\phi(\gamma) = \sum _i a_i t_i$. Then  we have
$$\gamma - \sum _i a_i \gamma_i  \in \ker  \phi.$$
This completes  the proof  of Lemma \ref{lem:direct1}.
\end{proof}

\begin{proposition}\label{prop:direct2}    Assume that $\Gamma$ is a  torsion free finitely generated  abelian  group  and $\phi: \Gamma\to  \R$ is a homomorphism.
Then  we  have  a ring  isomorphism $R((\Gamma, \phi)) =  R[\ker \phi] ((\phi (\Gamma)))$.
\end{proposition}
\begin{proof}   Using Lemma \ref{lem:direct1} we write  $\Gamma = \ker \phi \oplus \Gamma_\phi$.   By definition  any   element $\lambda \in R((\Gamma, \phi))$ can be written as follows
\begin{eqnarray*}
\lambda & = & \sum_{ij} \lambda_{(\alpha_i, \beta_j)} \cdot( \alpha _i + \beta_j)|\,  \lambda_{(\alpha_i, \beta_j)}  \in R, \, \alpha _i \in \ker \phi, \,  \beta _j \in \Gamma_\phi \text{ such that }\\
& & \text{ for any } C \in \R \: \#\{ [\alpha_i, \beta_j]|\, \lambda_{(\alpha_i, \beta_j)} \not = 0   \text{ and }  \phi (\beta_j) < C\} < \infty.  
\end{eqnarray*}

 It follows that given $\lambda\in R((\Gamma, \phi))$, for each $\beta_j\in \Gamma_\phi$   there  is only a finite number of $\lambda_{(\alpha_i, \beta_j)}$ such that
$\lambda_{(\alpha_i, \beta_j)} \cdot (\alpha _i + \beta_j)$ is a term  in $\lambda$. Hence
$$\sum _i\lambda_{(\alpha_i, \beta_j)} \cdot \alpha _i \in R [\ker \phi].$$  
Now we define a  map 
\begin{equation}
R((\Gamma, \phi))  \stackrel{\phi ^*}{\to} R[\ker \phi] ((\phi (\Gamma))), \,
\lambda \mapsto \sum_j (\sum_i\lambda_{(\alpha_i, \beta_j)} \cdot \alpha _i) \cdot \phi (\beta_j).   \label{eq:direct3}
\end{equation}
It is  straightforward to  verify  that $\phi^*$ is a  ring homomorphism, since $\Gamma$ is abelian. 

Since   the restriction of $\phi$  to
the subgroup $\Gamma_\phi$   is   a monomorphism,  from (\ref{eq:direct3}) we conclude that $\phi^*$  is a ring monomorphism.

%Identifying $\beta_j$ with the image $\phi(\beta_j)$  we obtain the following  monomorphism
Now let  $\delta \in R[\ker \phi] ((\phi (\Gamma)))$. We  write % Identifying $\beta_j$ with the image $\phi(\beta_j)$, we write
\begin{eqnarray} \delta & =  &\sum_i \delta_i  \cdot \phi(\beta_i )|\, \beta_i \in \Gamma_\phi, \,  \delta _i \in R[\ker \phi] \text{ such  that }\nonumber \\
& & \text{ for any } C \in \R\:   \#\{ \delta_i|\, \delta_i \not = 0   \text{ and }  \phi (\beta_i) < C\} < \infty. \label{eq:finite2}  
\end{eqnarray}
We write  $\delta_i = \sum _{ j =1}  ^{N(i)} \delta _{ij} \cdot \alpha _{ij}$, where  $\delta_{ij} \in R$, $\alpha_{ij} \in \ker \phi$. Then
$$\delta = \sum _{ij}\delta_{ij}\cdot\alpha_{ij} \cdot\phi( \beta _i). $$
Since for   each  $i$  the  number   of $\alpha_{ij}$ is finite, we obtain from (\ref{eq:finite2})
\begin{equation}
\#\{ \delta_{ij}|\, \delta_{ij} \not = 0   \text{ and }  \phi (\alpha_{ij} \cdot \beta_i) < C\} < \infty.   \label{eq:finite1}
\end{equation}
Using (\ref{eq:finite1}), we define  a map 
\begin{equation}
R[\ker \phi] ((\phi (\Gamma))) \stackrel{\phi_*}{\to} R((\Gamma, \phi))\, , \delta \mapsto  \sum_i  \sum _{ ij}   \delta_{ij} (\alpha _{ij}+\beta _i).\nonumber
\end{equation}
Set $\delta _{(\alpha_{ij}, \beta_i)}: = \delta_{ij}.$
Since $\Gamma$ is abelian,   $\phi_*$ is a ring  homomorphism. 
Observing that $\phi^*\circ \phi_* = Id$,  we   conclude that  $\phi^*$ is  an epimorphism. Hence $\phi^*$ is  an  isomorphism. This  proves  Proposition \ref{prop:direct2}.
\end{proof}

\begin{proposition}\label{prop:compono} The  ring $\Lambda ^R_{\theta, \om}$ is an integral domain.
%Consequently $\Lambda ^R_{\theta(-\tau, \tau), \om}$  is also an integral domain. 
\end{proposition}

\begin{proof}   By Proposition  \ref{prop:direct2}
$$\Lambda ^R _{\theta, \om} = R[\ker  \Psi_{\theta, \om}]((\Psi_{\theta, \om} (\Gamma ^0)).$$
Since  $\Gamma_0$ is a finitely generated  torsion free abelian group,  $\ker \Psi_{\theta, \om}$ is a  finitely generated  torsion free abelian group.
Hence the group ring $R[\ker \Psi_{\theta, \om}]$ is an integral  domain, see \cite[Chapter 3]{Passman1986}  for  a survey  on  Kaplansky's zero-divisor conjecture, or   see  Corollary  \ref{cor:integr} below. 
To complete   the proof  of  Proposition  \ref{prop:compono} we  need  the  following lemma, which has been formulated in \cite{HS1994}.

\begin{lemma}\label{lem:integral}  Assume  that  $R$  is  an integral  domain  and $\Gamma$ is a finitely  generated torsion  free  abelian group. Then  $R((\phi (\Gamma)))$ is an integral domain  for any   homomorphism $\phi: \Gamma \to \R$.
\end{lemma}

\begin{proof} Since  Hofer and Salamon omit  a  proof  of Lemma \ref{lem:integral}, which, in fact, can be constructed  from their  arguments in \cite[\S 4]{HS1994},    for  the sake  of reader's convenience  we  present here  a proof. % (The argument in this  proof can be  used  to prove the fact that $R[\ker \Psi_{\theta, \om}]$ is an integral  domain.)
Note that  $\phi(\Gamma)$ is a  finitely generated  torsion-free  abelian group. Let $m$ be  the rank of  $\phi(\Gamma)$. We identify   elements  of $\phi(\Gamma)$    with   $t ^k$,   where $t$ is a formal variable  and $k =  (k_1, \cdots, k_m ) \in \Z^m$.
We  say that    $k > k ' $ if  $ \phi  (t^k ) >  \phi( t^{k'})$.   This  defines  a  total ordering  on $\phi(\Gamma)$, which is compatible
with  the multiplication  on $\phi(\Gamma)$:
\begin{equation}
   k > k '  \LRA    k +  k ''  >  k ' + k ''  \text{ for  all } k'' \in  \Z^m. \label{eq:tord}
\end{equation}
 Then 
\begin{equation}
R ((\phi(\Gamma))) = \{ \sum_{i=1} ^\infty  a_i  t^{k_i}|\,  a_i \in  R, \,  a_1 \not = 0, \,  k_i  < k_{ i+1}\}.\label{eq:red1}
\end{equation}
%Since $R$ admits a natural embedding  in $K$, to  prove  Lemma \ref{lem:integral}, it suffices to show
%$K((\phi(\Gamma)))$ is  an integral domain.   
Assume  the opposite, i.e.   there exist  elements $A, B \in R((\phi(\Gamma)))$
such that  $A\cdot B = 0$  but $ A\not = 0$  and $B \not = 0$.  Using (\ref{eq:red1})  we write
$$ A = \sum_{i=0} ^\infty  a_i  t^{k_i}, \; B = \sum_{i=0} ^\infty  b_i  t^{l_i},$$
where
$$k_0 < k_1 < \cdots, \text{  and  } l_0 < l_1 < \cdots  $$
and  $ a_0 \not = 0$ and $b _0 \not = 0$.   Since  $A\cdot B = 0$  implies  $a_0\cdot b_0 = 0$,  taking into account the fact that $R$ is an integral domain,
we conclude that  $a_0 = 0 $  or $b_0 = 0$. This  contradicts  our assumption   that $a_0 \not = 0$ and $b_0 \not = 0$.  Hence  the  proof of Lemma \ref{lem:integral}  is completed.
\end{proof}

\begin{corollary}\label{cor:integr}  Assume that $R$ is an integral domain  and  $G$  is a finitely generated torsion free  abelian group. Then
$R[G]$ is  an integral domain.
\end{corollary}
\begin{proof}  Since  $G$ is   a finitely generated  torsion free  abelian group, there  is  a  monomorphism $\phi:G \to \R$.  Lemma \ref{lem:integral}  implies  that $R((G,\phi))$ is  an integral domain.
Since $\phi$ is  a monomorphism,  $R[G]$  is  a subring of $R((G, \phi))$.  It follows  that  $R[G]$ is  also   an integral domain.
\end{proof}

Clearly    Proposition \ref{prop:direct2}  follows from Lemma \ref{lem:integral}  and the  fact that $R[\ker \Psi_{\theta, \om}]$ is an integral  domain.
%The second assertion  is   a consequence of the first one.
\end{proof}

Recall that  the rank of a  module $L$ over an integral  domain $A$ is defined to be  the dimension  of  the vector space $ F(A) \otimes_A L$  over the   field  of fractions $F(A)$ of  $A$.
%Recall that $CFN_*(\Ht, J, R)$ is $\Z_{2N}$-graded.   Then, 
By Proposition \ref{prop:compono}, $\Lambda^R_{\theta, \om}$ is an integral domain. By Lemma \ref{lem:chain}
the chain group $CFN_k (\tilde H, R)$, and hence  Floer-Novikov homology group  $HFN_k (\tilde H, R)$  are left modules  over
$\Lambda ^R_{\theta, \om}$.  Thus we  define {\it the Betti numbers  $b_i (HFN_*(\tilde H, R))$} as follows
\begin{equation}
b_i(HFN_* (\Ht, J, R)  ) =  \dim_{F(\Lambda ^R_{\theta,  \om})}(F(\Lambda ^R_{\theta, \om})\otimes _{\Lambda _{\theta, \om}^R} HFN_i(\Ht, J, R) ).\label{eq:ucf1}
\end{equation}

\

\begin{lemma}\label{lem:field}  Assume  that  $\F = F (R)$  is  the field  of fractions  of an integral domain $R$. Then 

(i) $F(\Lambda ^R_{\theta, \om}) = F ( \Lambda ^\F _{\theta, \om})$.

(ii) $b_i(HFN_* (\Ht, J, R)  ) \le  \rk (CFN_i (\Ht, R))$.

(iii) $b_i(HFN_* (\Ht, J, R)  ) = b_i(HFN_* (\Ht, J, \F)  )$.

\end{lemma}

\begin{proof}  (i)     Since $R \subset \F$ we have  $F(\Lambda ^R_{\theta, \om}) \subset F ( \Lambda ^\F _{\theta, \om})$.
To   show that  $F(\Lambda ^R_{\theta, \om}) \supset F ( \Lambda ^\F _{\theta, \om})$  it suffices to  observe  that
 $F(\Lambda ^R_{\theta, \om})  \supset   F$.  This  proves the  first assertion  of Lemma \ref{lem:field}.

(ii)  The second assertion of  Lemma \ref{lem:field}  follows  from the universal  coefficient  theorem.     Since  the field of fractions
$F(\Lambda^R_{\theta, \om})$  is  a flat right  $\Lambda^R_{\theta, \om}$-module,  we obtain
$$b_i(HFN_* (\Ht, J, R)  ) = \dim_{\F(\Lambda^R_{\theta,\om})} H_i(F(\Lambda ^\F _{\theta, \om}) \otimes  _{ \Lambda ^R _{\theta, \om}}  CFN_*(\Ht, J, R))$$
$$\le \rk (CFN_i (\Ht, R)).$$

(iii)  The last assertion   of Lemma \ref{lem:field}  follows    from  the  first  assertion,  the universal coefficient theorem (\ref{eq:ucf2}), taking into  account  the identity
 $$ F(\Lambda ^\F _{\theta, \om}) \otimes  _{ \Lambda ^R _{\theta, \om}}  CFN_*(\Ht, J, R) = F(\Lambda ^\F _{\theta, \om}) \otimes  _{ \Lambda ^\F _{\theta, \om}}  CFN_*(\Ht, J, \F).$$
 This completes  the  proof  of   Lemma \ref{lem:field}.

\end{proof}

%%%%%%%%%%%%%%%%%%%%%%%%%%%%%%%%%%%%%%%%%%%%%%%%%%%%%%%%%%%%%%%%%%%%%%%%%%%%%%%%%%%%%%%%%%%%%%%%%%%%%%%%%%%%%%%%%%
%%%%%%%%%%%%%%%%%%%%%%%%%%%%%%%%%%%%%%%%%%%%%%%%%%%%%%%%%%%%%%%%%%%%%%%%%%%%%%%%%%%%%%%%%%%%%%%%%%%%%%%%%%%%%%%%%%
%%%%%%%%%%%%%%%%%%%%%%%%%%%%%%%%%%%%%%%%%%%%%%%%%%%%%%%%%%%%%%%%%%%%%%%%%%%%%%%%%%%%%%%%%%%%%%%%%%%%%%%%%%%%%%%%%%%%

\subsection{Admissible family of nondegenerate   Hamiltonian  functions}\label{subs:adm}

 Our goal in this  subsection is to prove Theorem \ref{thm:small}, which is an improvement of Lemma \ref{lem:chain} and  will be needed   for  our computation of  the  Betti numbers of $HFN_*(\Ht, R)$  later.
Let $\theta$ be a representative of $[\theta]$  and  let us pick  a function $h^\theta$ on $\Mt$ such that
\begin{equation}
  dh ^{\theta}  = \pi^*(\theta).\label{eq:theta}
\end{equation}

For $\lambda \in \R$ we set
$$C^\infty _{(\lambda)} (S^1 \times \Mt^{2n}) : = \{   f\in C^\infty (S^1 \times \Mt^{2n})|\: f = \lambda\cdot h^\theta  + \pi^*(\bar f) \}$$
  for  some $\bar f \in  C^\infty (S^1 \times M^{2n})$.
	
	The number  $\lambda$ will be called {\it the weight } of   $\Ht \in  C^\infty _{(\lambda)} (S^1 \times \Mt^{2n})$.

 Let
$$C^\infty _{(*)} (S^1 \times \Mt^{2n}) := \cup_{\lambda \in \R} C^\infty _{(\lambda)} (S^1 \times \Mt^{2n}). $$

Note that 
if $J$ is  regular,  $\Ht \in C^\infty_{(*)}(S^1 \times   \Mt)$ is  nondegenerate and $J$-regular, then the Floer-Novikov chain complex 
$(CFN_*(\Ht,R),\p_{(J, \Ht)}) $ is well defined.

\begin{remark}\label{rem:lambda}  (1)    If $\Ht \in  C^\infty _{(*)} (S^1 \times  \Mt)$   then  the  time-dependent Hamiltonian vector field $X_{\Ht}$ is
invariant under   the  covering transformation  group $\Gamma_1$ of  $\Mt^{2n}$.

(2) By Proposition   \ref{prop:iso}, if  $\Ht$ and $\Ht'$ belong   to the same  space 
$C^\infty _{(\lambda)} (S^1  \times \Mt^{2n})$,  the chain complexes  $(CFN_*( \Ht, R) ,\p_{(J,\Ht)})$  and $(CFN_*(\Ht', R), \p_{(J,\Ht')})$  are chain homotopy equivalent.
Hence  the Floer-Novikov homology  group $HFN_*( \Ht,  R)$  depends  only on the  number $\lambda \in \R$  such that 
$\Ht \in C^\infty _{(\lambda)} (S^1  \times \Mt)$.
\end{remark}

 For $\Ht \in C^\infty _{(\lambda)} (S^1\times \Mt^{2n})$ we set
$$
CFN_k^0 (\Ht,  R): =\{\sum_{\#(g_i) < \infty} g_i\cdot a_i |\,  g_i \in R, a_i \in Crit_k (\Aa_{\Ht})\}, 
$$
Then $CFN_k^0 (\Ht,   R)$ is a $R[\Gamma^0]$-module.

 For $0\le \tau_1 \le \tau_2 \in  \R$  and for $\Ht \in C^\infty _{(\lambda)} (S^1 \times \Mt^{2n})$  we set

$$\Lambda ^R_{\theta(\tau_1, \tau_2), \om} : =   \Lambda ^R_{\theta\cdot\tau_1, \om} \cap  \Lambda ^R_{\theta\cdot \tau_2, \om},$$
$$CFN_* ^{(\tau_1, \tau_2)} (\Ht,  R):=\Lambda ^R_{\theta(\lambda -\tau_1,\lambda + \tau_2), \om}\otimes _{R[\Gamma^0]} CFN_*^0 (\Ht,  R).$$

The chain complex $CFN_* ^{(\tau_1, \tau_2)} (\Ht,   R)$ will play important role  in  our    proof of  the Main Theorem, so  we   shall describe them
more carefully in Lemma \ref{lem:intersect}  below.
%We need     some new notations.

For  a function  $\Ht \in C^\infty _{(\lambda)} (S^1 \times \Mt)$  and for $\tau \in \R$  we set

\begin{equation}
\Ht ^ {(\tau)}: = \Ht +  \tau \cdot  h^ \theta .\label{eq:hgamma}
\end{equation}

%\color{black}

\begin{lemma}\label{lem:intersect}  We have
\begin{equation}
\Lambda ^R_{\theta (\tau_1, \tau_2), \om}  =\cap_{\tau \in [\tau_1, \tau_2]} \Lambda _{  \theta\cdot \tau, \om} ^R,\label{eq:intersect}
\end{equation}
\begin{eqnarray}
CFN_k ^{(\tau_1, \tau_2)}(\Ht,  R)& : = &\{ \sum \xi_{[\xt ,\tilde w ]} \cdot [\xt ,\tilde w ], \,  [\xt ,\tilde w ] \in \Ppt _k(\Ht ), \; \xi_{[\xt ,\tilde w ]} \in R|\nonumber \\
& &\text{ for all } c \in {\bold R}  \text {  for any } \tau \in (\tau_1, \tau_2)\nonumber \\
 & & \# \{ [\xt ,\tilde w ]|\,  \xi_{[\xt ,\tilde w ]} \neq 0 \, \&\, 
 \Aa_{\Ht^{(\tau)}}([\xt ,\tilde w ]) > c \}\nonumber \\
& &\text{ is finite } \label{eq:cfnk}
\end{eqnarray}
\end{lemma}
\begin{proof}  W.l.o.g. we assume  that $\tau_1 < \tau_2$.  Since $\Lambda ^R_{\theta (\tau_1, \tau_2), \om}  \supset\cap_{\lambda \in [\tau_1, \tau_2]} \Lambda _{  \theta\cdot\lambda, \om} ^R$, to   prove
(\ref{eq:intersect})     it suffices  to show that
\begin{equation}
 \Lambda^R_{\theta (\tau_1, \tau_2), \om}  \subset\cap_{\lambda \in [\tau_1, \tau_2]} \Lambda _{  \theta\cdot \tau, \om} ^R.\label{eq:inclusion1}
\end{equation}
Let  $\lambda \in \Lambda^R_{\theta (\tau_1, \tau_2), \om}$.    Then  we  write  (cf. (\ref{eq:novikov}))
\begin{eqnarray}
\lambda & = &\sum \lambda _g\cdot g, \, g \in \Gamma^\circ , \; \lambda _{g}\in R\,  \text{ s.t.  for all } c \in \R\nonumber \\
& & \# \{    g |\,  \lambda _{g} \neq 0\,  \&  \, \Psi_{\theta\cdot \tau_1, \om}(g) < c\} < \infty \text{  and } \nonumber\\
 &  &  \# \{ g | \,\lambda _{g} \neq 0\,  \&  \, \  \Psi_{\theta\cdot\tau_2, \om}(g) < c\} < \infty .\label{eq:novikov2}
\end{eqnarray}
Now assume that  (\ref{eq:inclusion1}) does not hold.  This implies  that (\ref{eq:intersect}) does not hold   for some  $c \in \R$ and  some $\tau \in (\tau_1, \tau_2)$, i.e.   there  exists 
$\lambda = \sum \lambda _g\cdot g\in \Lambda^R_{\theta (\tau_1, \tau_2), \om}$  such that
\begin{equation}
\#\{ g | \,  \Psi_{\theta\cdot\tau, \om}(g) < c\} = \infty. \label{eq:infiny}
\end{equation}
Set
$$  a = {\tau_2 -\tau\over \tau_2 -\tau_1}, \:  b = {\tau-\tau_1\over \tau_2 -\tau_1}.$$
Then  $0 < a, b <  1$ and
$$\Psi_{\theta\cdot \tau, \om} = a \Psi_{\theta\cdot \tau_1, \om} + b \Psi_{\theta\cdot\tau_2, \om}.$$
%Assume that  $g_1, \cdots,  g_\infty$   satisfy (\ref{eq:infiny}) and  $\lambda_{g_i} \not = 0$.
Since $\Psi_{\theta\cdot \tau, \om} (g) < c$  then  either
\begin{equation}
 \Psi_{\theta\cdot\tau_1, \om}(g_i) < \frac{c}{2a} \label{eq:A}
 \end{equation}
or
\begin{equation}
 \Psi_{\theta\cdot\tau_2, \om}(g_i) < \frac{c}{2b}.\label{eq:B}
 \end{equation}
 
 Denote by $\Gamma^0 (\lambda, \tau_1)$ (resp.  $\Gamma^0 (\lambda,  \tau_2)$)  the set  of    all $g \in \Gamma^0$ such that $ \lambda_g \not = 0$ and
 $g$ satisfies  (\ref{eq:A})  (resp.    (\ref{eq:B})).  The  above argument  implies
 $$\{ g |\, \lambda_ g \not = 0\,  \& \, \Psi_{\theta\cdot\tau, \om}(g) < c\} \subset \Gamma^0 (\lambda, \tau_1) \cup  \Gamma^0 (\lambda, \tau_2).$$
 Since  $\Gamma^0(\lambda, \tau_1)$ and $\Gamma ^0(\lambda, \tau_2)$ are finite sets  by (\ref{eq:novikov2}), it follows that (\ref{eq:infiny}) cannot happen.
 This  proves the first  assertion of Lemma \ref{lem:intersect}.
 
The second assertion of Lemma \ref{lem:intersect}  follows from  the first one. This completes the proof of Lemma \ref{lem:intersect}.
\end{proof}

% From now till the  end of  our paper we assume that $R = \Q$.
In  the remainder  of this subsection we shall      prove  a  family version of    \cite[Proposition 4.5]{Ono2005}, which     improves  Lemma \ref{lem:chain}. We assume that   $J$  is a regular   compatible  almost complex structure  on $(M^{2n}, \om)$.

First we describe  a special   set  of  admissible perturbations  of a  nondegenerate $J$-regular Hamiltonian  function
%Given  a   regular   compatible  almost complex  structure  $J$, 
%We     shall say that a  one-parameter  family $\{ \Ht _\chi, \; \chi \in [0,1]\}$
%of nondegenerate Hamiltonians $\Ht_\chi \in C^\infty _{(*)}(S^1 \times \Mt)$  is {\it admissible}, if 
%\begin{itemize}
%\item each $\Ht_\chi, \, s \in [0,1],$  is $J$-regular, 
% $\Pp(\Ht_\chi) = \Pp(\Ht_0)$   for all $\chi \in [0,1]$.
%\item $|d\Ht_\chi|_{C^0}$ is smooth  in  the variable $\chi$.
%\end{itemize}
%A family $\{ \Ht _\chi, \; \chi \in [0,1]\}$  will be called {\it admissible}, if  (2) and (3) hold.
 $\Ht_0 =  \lambda \cdot h^\theta  + \pi^* (H_0)$, where $H_0 \in C^\infty (S^1 \times M^{2n})$.
Let  $\{ x_1, \cdots   ,x_k\}$ be the   set  of  one-periodic  orbits  of the locally Hamiltonian  equation 
 associated to  $\lambda \cdot \theta$ and $H_0$.  Then $\{ \xt_i  = \pi^{-1} (x_i)\}$  are  one-periodic  orbits  of  the  Hamiltonian  flow generated by  $\Ht_0$.
Let $d$  denote the    distance on $\Mt$  induced  from the  Riemannian metric $\pi^* (g_J)$.  
We define   the distance  between   1-periodic orbits $\xt$ and $\yt$  as follows
  $$\rho (\xt, \yt) : = \int_0^1 d ( \xt (t), \yt(t))\, dt.$$
It is easy to see  that $\rho (\xt, \yt) = 0$ iff $\xt = \yt$  and
$$\rho(\xt, \yt) \le \rho (\xt, \zt) +  \rho (\zt, \yt) .$$ 
	Since   $\Ht_0$ is  nondegenerate,   there  exists  a   positive   number $\eps (\Ht_0) >0 $ such that
\begin{equation}
\frac{1}{4} \max _{ t \in S^1} d (\xt (t), \yt(t)) \ge \frac{1}{4} \rho (\xt,  \yt) > \eps(\Ht_0)  \label{eq:eps}
\end{equation}
  for distinct  orbits $\xt, \yt\in \Pp(\Ht_0)$.
Let $U_i, \, i \in [1, k],$  be the   $\eps$-tubular  neighborhood  of the graph $G_{x_i}$ of $x_i$ in $S^1 \times   M^{2n}$.
Then  $U_i$ are mutually disjoint.  Set $\tilde  U_i : = \pi ^{-1} (U_i)$.   

%We denote by $\Uu_\delta (\Ht)$  the subset   of  $ h \in C^\infty ([-\delta, \delta] \times   S^1 \times M^{2n})$  such that
%$$h_\chi |_{  U_i} = 0  \text { for  all }  \chi \in [-\delta, \delta].$$
%The space $\Uu_\delta (\Ht)$  is provided with  the 
%Following \cite{Floer1988}  we define  the following   Banach norm 
%

\begin{definition}\label{def:adm}  A family $\Ff := \{ \Ht_\chi|\, \chi \in [0, 1]\}$ of nondegenerate  Hamiltonian functions  in $C^\infty _{(*)} (S^1\times  M^{2n})$  will be called {\it   admissible}, if
\begin{enumerate}
%\item  $(\Ht_\chi)_{|  \tilde U_i}  = (\Ht_0)_{| \tilde U_i} $ for all  $\chi \in [0,1]$,
\item The map  $\chi \mapsto \Ht_\chi$ is  continuous   in  the $C^1$-topology  induced on $\Ff$,
\item $\Pp(\Ht _\chi) = \Pp (\Ht_0)$  for all $\chi \in [0,1]$.%The  Hamiltonian flow   generated  by  $\Ht_\chi$  has  the  same  one-periodic  orbits as   of $\Ht_0$. 
\end{enumerate}
\end{definition}

The  parameter space $[0,1]$ of an admissible family $\Ff$   can be replaced by any compact interval $[\delta, \delta'] \subset \R$, e.g. by reparametrization of $[\delta, \delta']$.  To make the   exposition  simple, we  consider in  this subsection only  admissible families with parameter $\chi \in [0,1]$.

\

Given  a function  $\Ht_\chi$  in an admissible    family  $\Ff$ of  nondegenerate  $J$-regular Hamiltonian   functions we set  
$$\Uu_c(\Ht_\chi) : = \{ \Ht _\chi + \pi^*(h_\chi)| \, h_\chi \in C^\infty (S^1 \times  M^{2n}), ||h_\chi||_\eps  <c,\,  (h_{\chi})_{| U_i} = 0 \forall  i\} $$
where $$||h||_{\eps} : = \sum _{ k =0} ^ \infty \eps_k ||h||_{ C^k ( S^1 \times M^{2n})}. $$
Here  $\eps_k > 0$ is a     sufficiently  rapidly  decreasing sequence  \cite{Floer1988},   see also \cite[\S 8.3]{AD2014}  for  a detailed discussion.  (In particular, we borrow
the condition $(h_{\chi})_{| U_i} = 0$  from  \cite[p.233]{AD2014}.)

\

We  also fix  a  vector space $\R^{N_0}$  and    an isometric  embedding $(M^{2n}, g_J)$ into   the  Euclidean  space $\R^{N_0}$. This  shall simplify  notations
of  different  norms on different  bundles over submanifolds in $M^{2n}$.

Further,  we  set  $\Pp(\Ff) : =\Pp (\Ht_0)$ and
$$\eps(\Ff) : = \eps (\Ht_0) $$
where $\eps(\Ht_0)$ is the constant in (\ref{eq:eps}).

The following Lemma is a family version of \cite[Lemma 5.2]{LO1995}. It contains   key  estimates (\ref{eq:deltau}), (\ref{eq:pu}), which  we shall exploit later in  Subsection \ref{subs:special}.

\begin{lemma}\label{lem:5.2.f}   Assume that  $\Ff:= \{ \Ht_\chi|\, \chi \in [0, 1]\}$  is an admissible  family of nondegenerate 
Hamiltonian functions in $C^\infty _{(*)}(S^1\times \Mt^{2n})$.
 There  exist a positive number $c: =c(\Ff) > 0$ and   a positive number   $\delta_1  = \delta_1 (\Ff)>0$ 
 such that   for any $\chi\in [0,1]$   the following   statement    hold. 

(i)  Let $\sigma(t)$   be a  smooth contractible loop on $\Mt^{2n}$  with 

$$\max_t  d(\sigma (t), \xt(t)) > \eps(\Ff)$$

 for any $\xt\in \Pp (\Ff)$.
Then  for any  $\Ht'_\chi \in  \Uu_c (\Ht_\chi)$  we have
 \begin{equation}
  || \dot \sigma  - X_{\Ht'_\chi}(G_\sigma)||_{L^2(S^1, \R^{N_0})}   > \delta_1(\Ff).  \label{eq:deltau}
\end{equation}		
(ii)   For any  $\chi \in [0,1]$ and  any  $\Ht'_\chi \in  \Uu_c (\Ht_\chi)$   we have   
\begin{equation}
\Pp (\Ht'_\chi) = \Pp (\Ff).\label{eq:pu}
\end{equation}
\end{lemma}

\begin{proof}  (i)  Assume the  opposite, i.e.  there  exist  the   following sequences
\begin{enumerate}
\item $\{c_j \in \R^+|\,  \lim _{j \to \infty}  c_j = 0\}$,
 \item $\{\chi(j) \in [0, 1]|\,   \lim_{ j \to \infty}  \chi (j) = \chi (\infty) \in [0,1]\}$,
\item $\{\Ht^j _{\chi(j)}\in \Uu_{c_j} (\Ht_{\chi (j)})\}$,
\item $\{\sigma_j \in \Ll M^{2n}|\, \max_t d(\sigma _j(t), x(t)) > \eps(\Ff) \text{  for  all } x \in \Pp (\Ff) \text{ and }$
	$$\lim _{j \to \infty}||\dot \sigma_j - X_{\Ht^j_{\chi(j)}}(G_{\sigma_j})||_{L^2(S^1, \R^{N_0}) } = 0\}.$$
\end{enumerate}	

 By Definition \ref{def:adm}(1),  
$$\lim_{j \to \infty} ||X_{\Ht^j_{\chi(j)}} - X_{\Ht_{\chi(\infty)}}||_{ C^0(S^1 \times M^{2n})} = 0.$$
 Hence
\begin{equation}
\lim _{j \to \infty}||\dot \sigma_j - X_{\Ht_{\chi(\infty)}}(G_{\sigma_j})||_{L^2(S^1, \R^{N_0}) } = 0. \nonumber  %\label{eq:lim1}
\end{equation}
By Lemma  5.1  in \cite{LO1995},   the last relation  implies that  a subsequence  of $\{ \sigma_j\}$  converges  to  some  contractible  orbit
$x\in \Pp (\Ht_{\chi(\infty)}) = \Pp (\Ff)$. This  is a contradiction, since $\max_t d(\sigma _j(t), x(t)) > \eps(\Ff)$.  The  proof  of Lemma \ref{lem:5.2.f}(i)  is completed.

\

(ii) Assume   that  there  is a contractible  orbit $\sigma (t)$ of  a Hamiltonian   function  $\Ht'_{\chi} \in \Uu_c(\Ht_\chi)$ such that $\sigma \not \in \Pp(\Ff)$.
If  the graph of $\sigma(t)$ belongs to      some   neighborhood $U_i \subset S^1 \times M^{2n}$ then $\sigma = \xt_i$, since  $(\Ht'_{\chi})_{| U_i}  = (\Ht_{\chi}) _{| U_i}$.
If  not then
$$\max_t (\sigma(t), \xt_i (t)) > \eps (\Ff)$$
for  any  $\xt \in \Pp (\Ff)$.  By the assertion  proved above  $\sigma(t)$ cannot be an orbit  of the flow   generated by $\Ht'_{\chi}$. We arrive at  a contradiction.  This completes the  proof of Lemma
\ref{lem:5.2.f}. 
\end{proof}

\begin{definition}\label{def:good} For  an admissible  family $\Ff:= \{ \Ht_\chi|\, \chi \in [0, 1]\}$  we  set
$$U(\Ff): = \cup_{\chi \in [0,1]}  U_{c(\Ff)} (\Ht_\chi).$$
where $c(\Ff)$ is  the constant in Lemma \ref{lem:5.2.f}.
We  call $U(\Ff)$  {\it  a  good neighborhood  of $\Ff$}.
\end{definition}
%$J$-admissible families  of  Hamiltonians have been first  considered in \cite{LO1995}.  They play  a key role  in our computing of the Betti numbers of the  Floer-Novikov  homology.

\begin{theorem}\label{thm:small}   Assume that $J$ is a regular  compatible almost complex structure
and  $\Ff: =\{\Ht_\chi \in  C^\infty_{(*)}(\Mt), \chi \in [0,1]\}$ 
   is an admissible  family   of nondegenerate   Hamiltonian functions. Assume that $\Ht_0$ is $J$-regular.
	
(i)	For each $\chi \in [0,1]$ the set  $\Uu^{reg}_{c}(\Ht_\chi)$  of $J$-regular Hamiltonian  functions  is dense  in $\Uu_c(\Ht_\chi)$   provided  with   the topology generated by  the Banach norm $||. ||_\eps$.

(ii)   There  exists  a  positive number  
	$\tau  : = \tau (\Ff) > 0$  with the following property.   Let  $\Ht_\mu \in  U(\Ff)$ is $J$-regular.  Then
for any  $[\xt, \wt] \in   Crit (\Aa_{\Ht_\mu})$ we have
$$\p_{(J,\Ht_\mu)} ([\xt, \wt]) \in CFN_*^{(-\tau, \tau)} (\Ht_\mu,  R) .$$  %\otimes _{R[\Gamma^\circ]}\Lambda ^R_{\theta(a\pm \delta), \om}. $$
Consequently  $(CFN_* ^{(-\tau, \tau)} (\Ht_\mu,  R), \p_{(J, \Ht_\mu)})$  is a chain complex.
\end{theorem}

\begin{proof}   (i)  The $J$-regularity   of $\Ht_0$  ensures that  any nondegenerate $\Ht_\mu \in U_c(\Ff)$   satisfies the  requirement (1) in Remark \ref{rem:comment1} for the $J$-regularity, since  $\Pp(\Ht_\mu) = \Pp(\Ht_0)$  by (\ref{eq:pu}). Clearly, the nondegeneracy  condition of $\Ht_{\mu} \in U_c(\Ff)$  defines a   open and dense  subset of  $U_c (\Ff)$.
 To prove    that  the   requirements   (2), (3) of the  $J$-regularity  also define a
dense  subset in $\Uu_c(\Ht_\chi)$ we use the standard   transversality argument   in the proof  of    Theorems  3.2, 3.3 in \cite{HS1994}, 
see also the proof of Theorems   3.2, 3.3 in \cite{LO1995}. 
 So we omit  the  proof.

(ii) Our  proof   of the second assertion uses   many ideas in the  proof of  \cite[Proposition 4.5]{Ono2005}.
 %Theorem  \ref{thm:small}   is essentially a variant   of  \cite[Proposition 4.5]{Ono2005}, taking into account  Lemma \ref{lem:intersect}.  For the reader convenience we      present
% a   proof in a slightly different  way than the  one in \cite{Ono2005}.
 
First we prove the following  two Lemmas containing     uniform estimates   for the  proof of  Theorem \ref{thm:small}.

We set
$$ e(\Ff) : = \min ( 4 \eps^2 (\Ff), \frac{\delta_1^2 (\Ff)}{2}).$$

\begin{lemma}\label{lem:onolem33}  (cf. \cite[Lemma 3.2]{Ono2005}) Suppose   that $\Ht_\mu \in U(\Ff)$  is $J$-regular,  $-\infty < R_1 < R_2 < \infty$,
$\xt_1, \xt_2 \in \Pp(\Ht_\mu)
= \Pp(\Ff)$  are distinct one-periodic orbits  and 
$u \in \Mm([\xt_1, \wt], [\xt_2, \vt], \Ht_\mu, J)$      satisfies $\max_t d(u (R_i, t), \xt_i (t)) \le \eps(\Ff)$. Then
$$\int_{R_1}^{R_2}\int_0^1 |\frac{\p u} {\p s}|^2 \, ds\, dt  > e(\Ff).$$
\end{lemma}

\begin{lemma}\label{lem:onolem34}(cf. \cite[Lemma 3.4]{Ono2005}, cf. \cite[Lemma 3.5]{LO1995}) Suppose   that $\Ht_\mu \in U(\Ff)$ is $J$-regular  and $\xt, \yt \in \Pp (\Ht_\mu) = \Pp (\Ff)$ are  distinct one-periodic orbits. 
For any  $u \in \Mm([\xt, \wt], [\yt, \vt], \Ht_\mu, J)$  we have
$$ E(u) = \int_{-\infty}^{\infty}\int_0^1 |\frac{\p u} {\p s}|^2 \, ds\, dt  \ge \frac{\delta_1(\Ff)}{2} \rho(\xt, \yt).$$
\end{lemma}

\

%\begin{remark}\label{rem:delta1} The number $\delta_1(\Ht)$ is    chosen    in  \cite{Ono2005}   to be equal the number $\delta_1 (\Ht)$   in Lemma \ref{lem:5.2}
%\end{remark}

\begin{proof}[Proof of Lemma \ref{lem:onolem33}] 
 Our   proof   is a refinement  of the proof  of  \cite[Lemma 3.2]{Ono2005}. %, which is  a refinement of the proof  of \cite[Lemma 3.5]{LO1995}, and make precise the meaning  of the constants  $e$ and $\eps$ in \cite[Lemma 3.2]{Ono2005}.
W.l.o.g.  we may assume that
\begin{enumerate}
\item $\max_t d(u(R_i, t), \xt_i (t)) = \eps (\Ff)$,
\item  For any $r \in (R_1, R_2)$  and    for $i = 1,2$ we have \\
$\max_{ t \in (R_1, R_2)} d(u(r, t), \xt_i (t)) > \eps(\Ff)$.
\end{enumerate}

\

{\it Case 1:  $R_2 - R_1 \le 1$.} Using the Cauchy-Schwarz  inequality  we obtain
$$\int_0^1\int_{R_1} ^{R_2}|\frac{\p u }{\p s} | ^2 \, ds dt \ge \int_0^1 (\int_{R_1}^{R_2} |\frac{\p u}{\p s}| ds)^2 dt \ge \int_0^1  (d(u(R_1, t),  u(R_2, t)))^2 \, dt. $$
Applying the  Cauchy-Schwarz   inequality again, we obtain  from the  above inequality
$$\int_0^1\int_{R_1} ^{R_2}|\frac{\p u }{\p s} | ^2 \, ds\, dt \ge \rho (u(R_1, -), u(R_2,- ))^2$$

$$\ge    (\rho (\xt_1, \xt_2) -\rho (\xt_1, u(R_1, -)) - \rho ( \xt_2, u(R_2, -))^2$$

$$\ge  (\rho(\xt_1 , \xt_2 ) - 2 \eps (\Ff))^2 > 4 \eps^2 (\Ff),  $$

since $\rho (\xt_i, u(R_i, -)) \le \max_t d(u(R_i, t), \xt_i (t)) = \eps (\Ff)$  and by (\ref{eq:eps})
$\rho(\xt_1 , \xt_2 ) - 2 \eps (\Ff) \ge 2\eps(\Ff)$.
\

Since  $4 \eps  ^2 (\Ff) \ge  e(\Ff)$, Lemma  \ref{lem:onolem33}  holds in Case 1.

\

{\it Case 2: $ R_2 - R_1  > 1$}.  Assume  that  Lemma \ref{lem:onolem33}   does not hold.  Since $R_2 - R_1 > 1$ there exists  $r \in  (R_1, R_2)$ such that
\begin{equation}
\int _0^1 |\frac{\p u(s, t) }{\p s} |^2(r, t)\, dt   < e (\Ff) \le \frac{\delta_1 ^2 (\Ff)}{2}.\label{eq:lem33}
\end{equation}
Since $u$ is a connecting orbit  associated with the  Hamiltonian $\Ht_\mu$ we obtain  from (\ref{eq:lem33})
\begin{equation}
|| \frac{\p u}{\p t} (r, .) -  X_{\Ht_{\mu}} (G_{u(r,.)})||_{L^2(S^1, \R^{N_0})} < \delta_1 (\Ff).\label{eq:ono33}
\end{equation}
By Lemma \ref{lem:5.2.f}, taking into  account  the assumption (2) at the beginning og the proof of Lemma \ref{lem:onolem33},  (\ref{eq:ono33}) cannot happen.  Hence  Lemma \ref{lem:onolem33}     also holds in  Case 2.
This completes  the  proof of Lemma \ref{lem:onolem33}.
\end{proof}

\begin{proof}[Proof  of Lemma \ref{lem:onolem34}] We repeat  the  proof of  \cite[Lemma 3.4]{Ono2005}, which is a refinement of   the  proof of Lemma \cite[Lemma 3.5]{LO1995},  and we make precise  the meaning of the constant  $\delta$  in the statement of  \cite[Lemma 3.4]{Ono2005}.

Let $u \in \Mm([\xt, \wt], [\yt, \vt], \Ht_\mu, J)$. Since $E(u) < \infty$, by Lemma \ref{lem:onolem33}  there are finitely many real  numbers  $-\infty < R ^{1-} <  R^{ 1+} < \cdots   R^{k-} <R^{ k +}  < + \infty $   and  one-periodic   solutions  $\xt_0 = \xt, \xt_1, \cdots,
\xt_k =  \yt$  such that 
\begin{enumerate}
\item $ \max_t  d(\xt_{i-1}  (t), u(R^{i-}, t))  = \max _t  d(\xt_t(t), u(R^{i+}, t)) = \eps (\Ff)$,
\item $\max_t (u(s, t), \zt(t)) > \eps(\Ff)$ for $s \in (R^{i-}, R^{i+})$  and $\zt \in \Pp (\Ff)$.
\end{enumerate} 
First we estimate
$$E ^{R^{i+}}_{R^{i-}} (u) : = \int_{R^{i-}} ^{R^{i+}} \int_0^1 | \frac{\p u }{\p s}| ^2 \, ds\, dt $$
$$ =  \int_{R^{i-}}^{R^{i+}} (\sqrt{\int_0^1 |\frac{\p u} {\p  s} (s, t) - X_{\Ht_\mu} (t, u (s, t))|^2 \, dt } ) ^2 \, ds.$$
Applying  the Cauchy-Schwarz inequality, we obtain
\begin{equation}
E ^{R^{i+}}_{R^{i-}} (u) \ge \frac{1}{R^{i+} - R^{ i -}} (\int_{R^{i-}} ^{ R^{i +}}\sqrt{\int_0^1 |\frac{\p u} {\p  s} (s, t) - X_{\Ht_\mu} (t, u (s, t))|^2 \, dt }\, ds )^2. \label{eq:cs}
\end{equation}
Combining   the  property (2)  with Lemma \ref{lem:5.2.f},  taking into account that $u$   is a connecting orbit  associated with the  Hamiltonian $\Ht_\mu$,  we obtain  from (\ref{eq:cs})
$$E^{R^{i+}}_{R^{i-}} (u) \ge \delta_1 (\Ff)\int_{R^{i-}} ^{R^{i +}}\sqrt{\int_0^1 |\frac{\p u}{\p s} |^2 dt}\, ds. $$
Applying  the Cauchy-Schwarz inequality again,  we obtain
$$E^{R^{i+}}_{R^{i-}} (u) \ge \delta_1 (\Ff)\int_{R^{i-}} ^{R^{i +}} \int_0^1|\frac{\p u}{\p s} | \, dt \, ds \ge \delta_1 (\Ff) (\rho (\xt_{i-1}, \xt_i ) - 2 \eps (\Ff)). $$
Using (\ref{eq:eps}), we  obtain
$$E^{R^{i+}}_{R^{i-}} (u) \ge \delta_1 (\Ff)(\rho (\xt_{i-1}, \xt_i)  -  2 \eps (\Ff)) \ge \frac{\delta_1 (\Ff)}{2} \rho (\xt_{i-1}, \xt_i).$$
Hence 
$$E(u) \ge  \sum_{i=1} ^k E^{R^{i+}} _{R^{i-}}(u) \ge  \frac{\delta_1 (\Ff)}{2} \rho(\xt, \yt).$$
This completes  the  proof of Lemma \ref{lem:onolem34}.
\end{proof}

\

{\it Continuation of the  proof of Theorem \ref{thm:small}} (ii)  Now assume that $\Ht_\mu\in U(\Ff)$ is $J$-regular.
We set  
\begin{equation}
\tau = \tau(\Ff): = \frac{\delta_1(\Ff)}{4 || \theta||_{C^0}}.\label{eq:alpha}
\end{equation}
Recall that $h^\theta$   is defined  in (\ref{eq:theta})  and $\Ht^{(\lambda)}_\mu$  is defined  in (\ref{eq:hgamma}). To  prove   Theorem \ref{thm:small} it suffices to show that  for any $C\in \R$  and  any $[\xt, \wt]$ with  $\mu([\xt, \wt]) = k$
we have
\begin{equation}
\# \{  u \in \Mm([\xt, \wt], [\yt, \vt])| \, \mu ([\yt, \vt]) =  k-1,\,  \Aa_{\Ht^{(\lambda)}_\mu} ([\yt, \vt]) >C\}< \infty   \label{eq:finite}
\end{equation}
for $\lambda =  \tau$  and   for $\lambda = -\tau$.
 We  write
\begin{eqnarray}
\Aa_{\Ht^{(\pm \tau)}_\mu} ([\yt, \vt]) =  \Aa_{\Ht ^{(\pm \tau)}_\mu}([\xt, \wt]) + \Aa_{\Ht_\mu} ([\yt, \vt])\nonumber\\
 +  \int_0^1  \pm \tau \cdot  h^{\theta}(\yt (t)) dt  - \Aa_{\Ht_\mu} ([\xt, \wt]) - \int_0^1\pm \tau \cdot  h^{\theta}(\xt(t))\, dt.\label{eq:decom1}
\end{eqnarray}
Taking into account  (\ref{eq:alpha})  and Lemma \ref{lem:onolem34},  we obtain
\begin{eqnarray}
|\int_0^1 \pm \tau \cdot  h ^{\theta}   (\xt (t))\, dt  - \int_0^1 \pm \tau \cdot h^{ \theta} (\yt (t))\, dt |\nonumber\\
\le |\tau\cdot\theta|_{ C^0}\cdot \rho (\xt, \yt) < \frac{E(u)}{2}.\label{eq:ono34}
\end{eqnarray}
%we obtain  from (\ref{eq:c1}), 
We obtain from (\ref{eq:decom1})  and (\ref{eq:ono34}), taking into  account the energy identity (\ref{eq:energyid})
%\begin{eqnarray*}
$$\Aa_{\Ht^{(\pm \tau)}_\mu} ([\yt, \vt])   > C $$
$$\LRA  \Aa_{\Ht_\mu} ([\yt, \vt])  >  C - \Aa_{\Ht ^{'(\pm \tau)}_\mu}([\xt, \wt]) + \Aa_{\Ht_\mu} ([\xt, \wt]) - \frac{E(u)}{2}$$
$$ >  C - \Aa_{\Ht ^{(\pm \tau)}_\mu}([\xt, \wt]) +  \Aa_{\Ht_\mu} ([\xt, \wt]) - \frac{\Aa_{\Ht_\mu} ([\xt, \wt]) - \Aa_{\Ht_\mu} ([\yt, \vt])}{2}$$
\begin{equation}
\LRA \frac{\Aa_{\Ht_\mu} ([\yt, \vt])}{2} \ge  C - \Aa_{\Ht ^{(\pm \tau)}_\mu}([\xt, \wt]) + \frac{\Aa_{\Ht_\mu} ([\xt, \wt])}{2}.\label{eq:filtr}
\end{equation}
Since  the RHS of (\ref{eq:filtr}), which   depends only on $C$  and on $[\xt, \wt]$,  is bounded form below,   there  is only finite  numbers  of $[\yt, \vt]$ that  satisfies  (\ref{eq:filtr}),
because $[\yt, \vt]$ enters in $\p_{J, \Ht_\mu}$.
This yields  (\ref{eq:finite}) and  completes the proof of Theorem \ref{thm:small}.
\end{proof}

From now on  we  abbreviate   the notation $(CFN_* ^{(-\tau, \tau)} (\Ht,  R), \p_{(J, \Ht)})$    as  $CFN_*^{(-\tau, \tau)} (\Ht, J, R)$.

\begin{proposition}\label{prop:rank}  Let $N$ be the minimal Chern number of  a compact symplectic manifold $(M^{2n}, \om)$  and $U(\Ff)$  a good neighborhood  of an admissible family $\Ff$ of  nondegenerate Hamiltonian functions
in $C^\infty_{(*)} (S^1 \times \Mt)$.
Assume that  $\Ht \in U(\Ff)$  is a $J$-regular  for some regular compatible almost complex structure $J$ on   $(M^{2n}, \om)$.   For any $i \in \Z_{2N}$ we have
$$ b_i (HFN_* (\Ht, J, R)  ) =   b_i (HFN_*  ^{(-\tau, \tau)}  (\Ht,  J, R)),$$
where $\tau = \tau (\Ff)$  is defined in (\ref{eq:alpha}).
\end{proposition}
\begin{proof}  
 %Assume  w.l.o.g. that  $\Ht \in C^\infty _{(1)} (S^1 \times \Mt)$. Then  %By definition we have %
%\begin{equation}
%b_i(HFN_* (\Ht, J, \Q)  ) =  \dim_{F(\Lambda ^\Q_{\theta,  \om})} HFN_i(\Ht, J, \Q)\otimes _{\Lambda _{\theta, \om}^\Q} F(\Lambda ^\Q_{\theta, \om}).\label{eq:ucf1}
%\end{equation}
%To compute the RHS of (\ref{eq:ucf1}) we  need the following.
%\begin{lemma}\label{lem:flat} The field of   fractions $F(\Lambda ^\Q _{\theta (\tau_1, \tau_2)})$ is  a flat left $\Lambda ^\Q _{\theta (\tau_1, \tau_2)}$-module.
%\end{lemma}
%\begin{proof}
%\end{proof}
%By Proposition \ref{prop:compono} $\Lambda _{\theta, \om}^R$  is  an integral domain.    Hence 
Since the field   of fractions  $F(\Lambda ^R_{\theta (\tau_1, \tau_2)})$ is  a flat right $\Lambda ^R _{\theta (\tau_1, \tau_2)}$-module,  
we obtain  from (\ref{eq:ucf1}), using the  universal   coefficient     theorem
\begin{equation}
b_i(HFN_* (\Ht, J, R)  )  = \dim _{F(\Lambda ^\R_{\theta,  \om})} H_i(F(\Lambda ^R_{\theta, \om})\otimes _{\Lambda _{\theta, \om}^R}CFN_*(\Ht, J, R) ).\label{eq:ucf2}
\end{equation}
Using Theorem \ref{thm:small}, we obtain  from (\ref{eq:ucf2})
$$b_i(HFN_* (\Ht, J, R)  ) =  \dim _{F(\Lambda ^R_{\theta,  \om})} H_i(F(\Lambda ^R_{\theta, \om})\otimes _{\Lambda^R _{\theta (-\tau, \tau) , \om}}(CFN_* ^{(-\tau, \tau)}(\Ht, J, R) )$$
$$=\dim _{F(\Lambda ^R_{\theta(-\tau, \tau),  \om})} H_i(F(\Lambda ^R_{\theta(-\tau, \tau), \om})\otimes _{\Lambda^R_{\theta (-\tau, \tau), \om}}(CFN_* ^{(-\tau, \tau)}(\Ht, J, R) )$$
$$ =\dim _{F(\Lambda ^R_{\theta(-\tau, \tau),  \om})}(F(\Lambda ^R_{\theta(-\tau, \tau), \om}) \otimes _{\Lambda^R_{\theta (-\tau, \tau) , \om}} HFN_i ^{(-\tau, \tau)}(\Ht, J, R)) $$
%$$ = \rk_{\Lambda _{\theta (-\tau, \tau) , \om}} HFN_i ^{(-\tau, \tau)}(\Ht, J, R)\otimes _{\Lambda _{\theta (-\tau, \tau), \om}} F(\Lambda ^R_{\theta(-\tau, \tau), \om})) $$
$$ = b_i(HFN_* ^{(-\tau, \tau)} (\Ht, J, R)  )  .$$
This completes  the proof of Proposition \ref{prop:rank}.
\end{proof}

%(It is an    analogue of  Lemma 2.5 in \cite{Farber2004},  p. 46)

%%%%%%%%%%%%%%%%%%%%%%%%%%%%%%%%%%%%%%%%%%%%%%%%%%%%%%%%%%%%%%%%%%%%%%%%%%%%%%%%%%%%%%%%%%%%%%%%%%%%%%%%%%%%%%%%%%%%%%%%%%%%%%%%%%%%%%%%%%%%%%%%%%%%%%%%%%%%%%%%%%%%%%%%%%%%%%%%%%%%%%%%
%%%%%%%%%%%%%%%%%%%%%%%%%%%%%%%%%%%%%%%%%%%%%%%%%%%%%%%%%%%%%%%%%%%%%%%%%%%%%%%%%%%%%%%%%%%%%%%%%%%%%%%%%%%%%%%%%%%%%%%%%%%%%%%%%%%%%%%%%%%%%%%%%%%%%%%%%%%%%%%%%%%%%%%%%%%%%%%%%%%%%%%%%%%
%%%%%%%%%%%%%%%%%%%%%%%%%%%%%%%%%%%%%%%%%%%%%%%%%%%%%%%%%%%%%%%%%%%%%%%%%%%%%%%%%%%%%%%%%%%%%%%%%%%%%%%%%%%%%%%%%%%%%%%%%%%%%%%%%%%%%%%%%%%%%%%%%%%%%%%%%%%%%%%%%%%%%%%%%%%%%%%%%%%

\subsection{Invariance of  the Betti numbers of   Floer-Novikov  chain complexes}\label{subs:special}
In this subsection we   assume  that  $\F$  is a field. 
 The goal of this subsection is to   prove the following.

\begin{theorem}\label{thm:comp}   The  Betti  numbers   $b_i (HFN_* ( \Ht, \F))$   do not depend on the choice of $\Ht \in  C^\infty _{(*)} (S^1\times \Mt^{2n})$.
\end{theorem}

\begin{proof} By Remark  \ref{rem:lambda}  the Betti  numbers $b_i (HFN_* (\Ht,\F))$  depend  only on  $\lambda$,  where $\Ht \in C^\infty _{(\lambda)} (S^1 \times  \Mt^{2n})$. Thus, to  prove Theorem \ref{thm:comp},  it suffices  to prove   the following.

\begin{proposition}\label{prop:bettismall}   Assume  that $J$  is regular. Let  $\Ht \in C^\infty _{(\lambda)} (S^1 \times \Mt)$  be  nondegenerate and $J$-regular.  Then there  exists  a number $\delta = \delta(\Ht) > 0$     such that  for any $\chi \in (-\delta, \delta)$  there is  a nondegenerate  $J$-regular  Hamiltonian function $\Ht_{\mu(\chi)} \in  C^\infty _{(\lambda +\chi)} (S^1 \times \Mt^{2n})$  with
$$b_i (HFN_* (\Ht,\F)) = b_i (HFN_*(\Ht_\chi ,\F))$$
for all $i \in \Z_{2N}$.
%and  $\{\Ht_\eps  = \Ht +  \tilde  h_\eps\}$  is a  J-admissible    family of  Hamiltonians. %are  also   nondegenerate and $J$-regular.
\end{proposition}

\begin{proof} Let $\Ht \in C^\infty _{(\lambda)} (S^1 \times \Mt)$. We shall  construct  an admissible  family $\Ff \ni \Ht$  and   find an  open  interval $(-\delta, \delta)$  such  that
for any $\chi \in (-\delta, \delta)$  there is  a function $\Ht_{\mu(\chi)}\in  C^\infty _{(\lambda +\chi)} (S^1 \times \Mt^{2n}) \cap U(\Ff)$  which satisfies 
the required  property in  Proposition \ref{prop:bettismall}.  This will be done in  6 steps.

 Since $\Ht \in C^\infty _{(\lambda)} (S^1 \times \Mt)$,  there  is a function $H \in C^\infty (S^1\times \Mt^{2n})$  such  that %
$$ \Ht -  \lambda\cdot  h^\theta  = \pi^*(H).$$

%Clearly, the  solution  of  the locally Hamiltonian  equation (\ref{eq:lham3})
 %associated to  $\lambda \cdot \theta$ and $H$ is the projection of  the  solution  of the  Hamiltonian equation  on $\Mt$ associated  with $\Ht$. 
 
%Assume that $J$  is a  regular compatible almost complex structure  on $M^{2n}$.
% Let $d$  denote the    distance on $\Mt$  induced  from the  Riemmanian metric $\pi^* (g_J)$.  
%We define   the  average distance  between   1-periodic orbits $\xt$ and $\yt$  as follows
 % $$\rho (\xt, \yt) : = \int_0^1 d ( \xt (t), \yt(t))\, dt.$$

\
	
\underline{Step 1}.  In  this step  we  construct  a  ``linear  part"  of  the desired  admissible family $\Ff \ni \Ht$  of Hamiltonians   for the  proof of   Proposition \ref{prop:bettismall}.  The main point  of this step is Lemma \ref{lem:eta}.

 Denote by  $ U_i : = U_\eps (G_{x_i})$ the open $\eps$-tubular neighborhood  of the  graph $G_{x_i} \subset   S^1 \times  M^{2n}$ of the  periodic   solution
$x_i \in \Pp (\Ht)$, where   $ \eps = \eps  (\Ht)$  satisfies  the  inequality  in (\ref{eq:eps})
\begin{equation}
\frac{1}{4} \max _{ t \in S^1} d (x_i (t), x_j(t)) \ge \frac{1}{4} \rho (x_i , x_j) > \eps(\Ht)  \nonumber
\end{equation}
if $i \not = j$.  %Moreover, 

%Since    $U_i$ are   disjoint tubular neighborhoods  of  contractible  orbits, it is not hard to obtain the following 
% the  existence  of   $\delta$ and  $\tilde h$   that satisfy       the conditions in   the following  Lemma.

% W.l.o.g  we assume  that  the closure  of $ U_i$  lies in   $V_i$  which   is also  an $\eps'$-tubular neighborhood  of $G_{x_i}$, moreover
%$V_i \cap V_j = \emptyset$, if $i \not = j$. %Then we  choose  open subsets $  W_i  \subset V_i$ such that
%$$U_i  \subset  W_i \text  { and }  \overline{W_i} \subset   V_i.$$
Let $p : S^1 \times  M^{2n} \to M^{2n}$ and $q: S^1 \times M^{2n} \to S^1$ denote  the projections onto the second and the first component  respectively. We  set
$$T^{0,1} (S^1 \times  M^{2n})  : = p ^* ( T^*M^{2n}), \:  \:   T^ {1, 0}(S^1  \times M^{2n}) : =   q ^* (T^* S^1).$$
Then we  have $T^* (S^1 \times M^{2n}) =  T^{0,1} (S^1 \times  M^{2n}) \oplus T^ {1, 0}(S^1  \times M^{2n})$.   This  yields a direct  decomposition
$$ \Om ^1 (S^1 \times M^{2n} )  =\Om ^{0,1} (S^1 \times M^{2n}) \oplus  \Om ^{1, 0} (S^1 \times M^{2n}),  $$
where
$$\Om ^{0,1} (S^1 \times M^{2n}) : = \{  \zeta \in \Om ^1 (S^1 \times  M^{2n})|\,  \zeta(t, x) \in T^{0,1} (S^1 \times M^{2n})\},$$
$$\Om ^{1,0} (S^1 \times M^{2n}) : = \{  \zeta \in \Om ^1 (S^1 \times  M^{2n})|\,  \zeta(t, x) \in T^{1, 0} (S^1 \times M^{2n})\}.$$

\

For a function $H \in C^\infty (S^1 \times  M^{2n})$  denote by $d_x H$ the  projection  of $dH$ on  the  component  $\Om ^{0,1} (S^1 \times  M^{2n})$.
Denote by $p $ the natural projection $ S^1 \times M^{2n}  \to M^{2n}$.

\begin{lemma}\label{lem:eta}  There  exists  a    1-form $\eta\in \Om^{0,1} (S^1 \times M^{2n})$ such that   the following conditions hold:\\
1)  $\eta   - p^*(\theta)   =  d_x H$,   for some  $H \in C^\infty (S^1 \times  M^{2n})$, \\
2) $  \eta_{| U_i} = 0 $ for all $i \in [1,k]$.
\end{lemma}
\begin{proof}    Lemma \ref{lem:eta}  has been used in \cite{LO1995} without  a (detailed)  proof. For  the reader's convenience
we  present    a  detailed  proof here.
 Since $d_x p^* (\theta) =  d_x \theta =0$, and  $x_i(t)$ is a contractible   curve  in $M^{2n}$,  there  exists  a function $H_i  \in C^\infty (U_i)$ such that  
\begin{equation}
p^* (\theta) _{| U_i}   =  d_xH_i  .  \label{eq:dx1}
\end{equation}
Since $U_i$   are  mutually  disjoint, there   exists a function $H \in C^\infty (S^1 \times   M^{2n})$ such that
\begin{equation}
H _{ | U_i}   = H_i .\label{eq:dx2}
\end{equation}
%Let $H  (t, x)\in   C^\infty (S^1 \times  M^{2n})$  satisfy  the following property  for all $ i \in [1, k]$.
%\begin{equation}
%H(t, x) _{| U_i}  = l (t, x) |_{ U_i} \text {and }  H(t, x)_{| V_i \setminus   W_i } = 0.\label{eq:dx2}
%\end{equation}
Now we set  $\eta = p ^* (\theta) -  d_x H $.  Then   $\eta$ satisfies the  first    condition in Lemma \ref{lem:eta}.
  By (\ref{eq:dx1}), (\ref{eq:dx2}) we have
$$\eta_{|  U_i} =  d_x H_i  - d_x H_i = 0.$$
Thus  $\eta$  also satisfies  the second   condition  of Lemma  \ref{lem:eta}. This completes  the  proof  of Lemma \ref{lem:eta}.
\end{proof}

\underline{Step 2}.  In this  step, using  $\eta$ in Lemma \ref{lem:eta}, we     construct  ``the action"  of   the  desired    admissible family $\Ff$.
The main point  of this step is  Lemma \ref{lem:same}.

First we choose  a positive number $\delta_1 =\delta_1(\Ht)$ from  following  Lemma, which is a special  case of Lemma \ref{lem:5.2.f}.

\begin{lemma}\label{lem:5.2} There  exists  a positive number   $\delta_1 = \delta_1 (\Ht)> 0$  
 such that   
$$|| \dot \sigma  - X_{\Ht}(G_\sigma)||_{L^2(S^1, \R^{N_0})}   > \delta_1$$
  for any    loop $\sigma(t)$ in $\Mt^{2n}$  satisfying $\max_t  d(\sigma (t), x(t)) > \eps(\Ht)$ for any $x\in \Pp (\Ht)$.
\end{lemma}

\
Let $\eta_t : =  \eta (t, -)$.
Denote by $\phi_t ^{\eta}$  the symplectic flow  on $M^{2n}$  that is generated by   the time-depending symplectic  vector  field  $L_{\om} ^{-1} (\eta_t)$  with $\phi_0^{\eta} = Id$.  

\

Now we  choose   a small positive number   $0<\delta_2 = \delta_2 (\Ht, \eta) < \delta_1/3$ such that
\begin{equation}
\| \pi_*( X_{\Ht_t})  - d\phi_t ^{c \cdot \eta} (\pi_* (X_{\Ht_t})) \|_{C^0 ( M^{2n})}   < \frac{\delta_1}{3}\label{eq:est1}
\end{equation}
and
\begin{equation}
\|X_{c\cdot \eta_t}\|_{ C^0(M^{2n}) }< \frac{\delta_1}{3}\label{eq:est2}
\end{equation}
for any $c \in [-\delta_2, \delta_2]$ and any $t \in [0,1]$.  The number  $\delta_2$ exists, since   $[0,1] \times M^{2n}$ is compact and    $d\phi_t  ^{0\cdot \eta} = Id$  for all $t \in [0,1]$. %(\blue{We can use  the stronger  $C_0$-norm   instead of the $L^2$-norm in (\ref{eq:est1})}).

\

Denote by $\varphi_t$ the symplectic flow on $M^{2n}$ generated by  the time-dependent vector  field $\pi_*(X_{\Ht_t})$.

\begin{lemma}\label{lem:same}  The    symplectic flows   $\varphi_t$ and   $\phi_t ^{c\cdot \eta}\circ \varphi_t$  have the same    contractible one-periodic  orbits
for  all $c\in [-\delta_2, \delta_2]$.  
\end{lemma}

\begin{proof}  Lemma \ref{lem:same}    and    our  proof   stem   from  analogous      arguments  in \cite[\S 3.2]{Ono2005}.
Assume the opposite,  i.e.  there are  a number  $c\in [-\delta_2, \delta_2]$ and a  contractible  one-periodic  orbit $\sigma (t)$  of  the flow $\phi_t^{c\cdot \eta}\circ \varphi_t$
which is not  a one-periodic orbit of $\varphi_t$.  We  abbreviate $L_{\om} ^{-1} (\eta_t)$ as $X_{\eta_t}$.   We  compute
\begin{equation}
{d \over dt}(\phi_t^{c\cdot \eta}\circ \varphi_t (x) )= d\phi_t ^{c\cdot \eta}  (\pi_*(X_{\Ht_t}) (\varphi_t (x)) + X_{c\cdot\eta_t} (\phi_t^{c\cdot \eta}\circ \varphi_t (x))\label{eq:vector}
\end{equation}
 Now let $x = \sigma(0)$.  If   $(t,\sigma (t) )\subset  U_i$ then  by (\ref{eq:vector})
\begin{equation}
{d \over dt}(\phi_t^{c\cdot \eta}\circ \varphi_t (\sigma(0))) = d\phi_t ^{c\cdot \eta}  (\pi_*(X_{\Ht_t}) (\varphi_t (x)).\label{eq:vector2}
\end{equation}
Since  $\eta_{| U_i}  = 0$, it is not hard  to  conclude    from (\ref{eq:vector2}) that 
$$ \sigma (t) : = \phi_t^{c\cdot \eta} (\varphi_t (\sigma(0)) = \varphi_t (\sigma (0)).$$
Hence $\sigma(t)$ is also an  one-periodic orbit   of $\varphi_t$, which  is a contradiction. This implies that
 %Denote by  $(\widetilde \phi_t^{c\eta})$ the  induced  flow   on $S^1 \times  M^{2n}$ that is generated by  time-independent vector field $X_{\Ht (t, x)}$. Since $\eta|_{U_i} = 0 $ for  each neighborhood $U_i
%$\subset  S^1 \times M^{2n}$  the restriction of    to $U_i\subset  S^1 \times M^{2n}$  is the identity  for any $t \in  [0,1]$. 
%This implies that   the   graph of the orbit
$(t,\sigma (t))$    does not lie in $U_i$  for any $i$. Hence
\begin{equation}
\max_t d(\sigma(t),  x _i (t))  > \eps(\Ht) \nonumber
\end{equation}
 for any $i$.  Combining with Lemma \ref{lem:5.2}  we obtain
\begin{equation}
\|\pi_*( X_{\Ht_t}) (\sigma (t)) - \dot \sigma (t)\|_{L^2(S^1, \R^{N_0})} > \delta_1(\Ht).\label{eq:est3}
\end{equation}
%Since  the flow   $\phi_t ^{c\cdot  \eta} \circ \varphi_t$ is generated by the  time-dependent vector field
Using (\ref{eq:vector}) we obtain
  the following inequalities,  taking into account  the inequalities (\ref{eq:est3}), \ref{eq:est2}), (\ref{eq:est1}), 
$$ ||{d\over dt} (\phi_t ^{c\cdot  \eta} \circ \varphi_t)(\sigma (0)) - \dot \sigma(t)||_{L^2 (S^1, \R^{N_0})} $$
$$=||d\phi_t ^{c \cdot \eta} (\pi_*(X_{\Ht_t})(\varphi_t(\sigma(0))) + X_{c\cdot \eta_t} (\sigma (t))  -\dot\sigma(t)||_{L^2(S^1, \R^{N_0})}  $$ 
$$ \ge ||\pi_*( X_{\Ht_t})(\sigma(t)) - \dot \sigma(t)||_{L^2(S^1, \R^{N_0})} - || X_{c\cdot  \eta_t}||_{C^0 (M^{2n})} $$
$$-|| \pi_*( X_{\Ht_t})(\sigma(t))  - d\phi_t ^{c \cdot \eta} (\pi_* (X_{\Ht_t}))(\sigma(t))||_{L^2(S^1, \R^{N_0})}  $$
$$  \ge \delta_1 - \frac{\delta_1}{3} -\frac{\delta_1}{3} > 0 .$$
 We arrive  at  a contradiction.  This completes the  proof of Lemma \ref{lem:same}.
\end{proof}

\underline{Step 3}  In this  step we construct  the desired admissible family $\Ff \ni \Ht$.  The main point  of this  step
is  Lemma  \ref{lem:hregular}.

  Assume that $\chi \in [-\delta_2, \delta_2]$.   Then   for all  $t\in [0,1]$% Note that  the flow   $\phi_t ^{\chi \cdot  \eta} \circ \varphi_t$ is generated by the time-dependent vector field
%$d\phi_t ^{\chi\cdot \eta} (\pi_*(X_{\Ht_t}))  + X_{\chi \cdot \eta_t}$. Clearly %there  exists  a $d_x$-closed 1-form $\xi \in \Om^{0,1} (S^1 \times M^{2n})$ such that
$$[L_\om^{-1}(d\phi_t ^{\chi\cdot \eta} (\pi_*(X_{\Ht_t} (\phi_t (x))  - \pi_*(X_{\Ht_t}) (\phi_t^{\chi\cdot \eta}\circ \varphi_t( x))]  = 0  \in H^1 (M^{2n}, \R).$$
%By  (\ref{eq:vector})  the  Calabi invariant  of the   time $t$-map $\phi_t  ^{\chi\cdot \eta}\cdot \varphi_t$, for any $t \in [0,1]$  is  $[(\lambda + \chi)\theta] \in H^1 (M^{2n}, \R)$. 
Hence  there   exists  a  unique  function  $  h^0_\chi  \in C^\infty (S^1  \times  M^{2n})$ such that   for   a given  point $x_0 \in M^{2n}$
$$h^0_\chi (x_0)  = 0$$
and $\Ht + \chi\cdot h^\theta +\pi^*( h^0_\chi)$  generates the  Hamiltonian   isotopy  whose  projection  on $M^{2n}$   is the isotopy  $\phi_t ^{\chi \cdot  \eta} \circ \varphi_t$.
We  set  
\begin{eqnarray}
\tilde h^0_\chi: = \chi\cdot h^\theta + \pi^*(h^0_\chi),\label{eq:htilde}\\
\Ht_\chi: =  \Ht +  \tilde h^0_\chi. \label{eq:Htilde}
\end{eqnarray}
%\begin{lemma}\label{lem:isotopy} The  symplectic  isotopy $\phi_t^{\eps\cdot \eta}\circ \varphi_t$ is generated by a time-dependent ($d_x$-) closed  1-form $\xi \in  \Om ^{0,1} (S^1 \times M^{2n})$   such that
 % for all $t $ we have $[\xi_t] \in \eps \cdot [\theta] \in H^1 (M^{2n}, \R)$.
%\end{lemma}
%\begin{proof} We compute
%$$\frac{d}{dt}\phi_t^{\eps\cdot \eta}\circ \varphi_t(x) = d\phi_t ^{\eps \cdot \eta} (\pi_*(X_{\Ht})(\sigma(t)) + X_{\eps\cdot \eta_t}.    $$
%Since  $d\phi_0^{\eps \cdot \eta}$  is  the identity,    for all $t \in [0,1]$ we have
%$$[L_\om (d\phi_t^{\eps\cdot \eta} (\pi_*(X_{\Ht})) -\pi_*(X_{\Ht}) ) ]= 0 \in H^1 (M, \R).$$
%Taking into account  Lemma \ref{lem:eta}.1,  we conclude that $\tilde h_\eps^0 \in C^\infty_{(\eps)} (S^1 \times \Mt)$.
%\end{proof}
%Using  the   following inequality
%\begin{equation}
%| d_x\tilde  h_\chi^\theta (t, x) |_{C^0}  <    |d\phi_t ^{\chi \cdot  \eta} - Id|_{C^0} \cdot |X_{\Ht_t}|_{C^0} + |\chi \cdot (L_\om ^{-1} ( \eta_t)) |_{C^0} \label{eq:conti}
%\end{equation}
%for all $t \in  [0,1]$,
%there   exists   a positive numbers $\delta < \delta_2$  and $ \tau > 0$ such  that
%\begin{equation}
%| d\tilde  h_\chi^\theta  |_{C^0}  < \frac{\delta_2}{2}  -\tau \label{eq:est5}
%\end{equation}
 
%for all  $\chi \in  [-\delta, \delta] \subset [-\delta_2, \delta_2]$.

%By Lemma \ref{lem:same} $\Ht  +  \tilde h_\chi^0$ is  a nondegenerate  Hamiltonian.  %We set 
%$$\Ht_\eps : = \Ht + h_\eps.$$

From Lemma \ref{lem:same} we obtain immediately the following, observing that the  nondegeneracy of $\Ht$ is an open property.

\begin{lemma}\label{lem:hregular} There is a positive number  $\delta_3 (\Ht) \le  \delta_2 (\Ht, \eta)$ such that the following  statement holds.
The  family  $\Ff (\Ht) : = \{ \Ht_\chi \in  C^\infty_{(\lambda +\chi)} (S^1  \times  \Mt^{2n})|\,  \chi \in [-\delta_3 (\Ht), \delta_3 (\Ht)]\}$ is  an admissible 
family  of nondegenerate  Hamiltonian  functions.
\end{lemma}

\

\underline{Step  4}.  
In this  step   we   shall choose   first candidates  for  $\delta(\Ht)$  and  a $J$-regular   nondegenerate  Hamiltonian $\Ht_{\mu(\chi)}$ for   the  proof of Proposition \ref{prop:bettismall}.
 
Set 
$$\delta_4  = \delta_4 (\Ht) : = \min \{  {\delta_3(\Ht)\over 2}, {\tau (\Ff)\over 2} \}.$$
Then 
\begin{equation}
\delta_4 \le \min \{ \tau(\Ff) -\chi, \tau (\Ff) + \chi \} \text { for  all } \chi \in  [-\delta_4, \delta_4]. \label{eq:deltah}
\end{equation}

\begin{lemma}\label{lem:uniform} For any  $\chi \in  [-\delta_4 ,  \delta_4]$     and  any $J$-regular  Hamiltonian  function 
$\Ht_{\mu(\chi)} \in  U_c^{reg} (\Ht_\chi)$ the  following   assertion holds.
 Let $[\xt, \wt]  \in  Crit (\Aa_{\Ht_{\mu(\chi)}})$. Then 
$$\p_{J, \Ht_{\mu(\chi)}}  (\xt, \wt) \in  \Lambda^ \F _{\theta(\lambda - \delta_4, \lambda + \delta_4), \om} \otimes_{\F[\Gamma^0]} CFN_* ^0 (\Ht_{\mu(\chi)}),  \F).$$ 
Furthermore
$$b_i (CFN_* (\Ht_{\mu (\chi)}, J, \F)  =  b_i (\Lambda^ \F _{\theta(\lambda - \delta_4, \lambda + \delta_4), \om} \otimes_{\F[\Gamma^0]} CFN_* ^0 (\Ht_{\mu(\chi)}),  J, \F).$$ %Consequently    $(CFN_* ^{(\eps, \beta(\delta) -\eps)} (\Ht_\eps,  J), \p_*)$ is
%chain complex, whose underlying   ring is  $\Lambda_{\theta(1, 1 + \beta(\delta)), \om}^\Q$.
\end{lemma}

\begin{proof} Let $\Ht_{\mu(\chi)} \in  U_c^{reg} (\Ht_\chi)$.  We  set $\tau: = \tau  (\Ff)$. By Theorem \ref{thm:small} we  have
\begin{equation}
 \p_{J, \Ht_{\mu(\chi)}}  (\xt, \wt) \in  CFN_*^{-\tau, \tau}(\Ht_{\mu (\chi)}, \F).\label{eq:red1}
 \end{equation}
 Since  $\Ht_{\mu (\chi)} \in C^\infty _{(\lambda + \chi)}  (S^1 \times  \Mt ^{2n})$, by definition we have
 
\begin{equation}
CFN_*^{-\tau, \tau}(\Ht_{\mu (\chi)}, \F) = \Lambda ^\F _{\theta ( \lambda + \chi - \tau, \lambda+\chi + \tau), \om} \otimes _{\F[\Gamma ^0]} CFN_* ^ 0 (\Ht_{\mu(\chi)}, \F). \label{eq:uni1}
\end{equation}
 
 From (\ref{eq:deltah}) we obtain
 \begin{equation}
 \lambda + \chi - \tau \le  \lambda - \delta_4 \le \lambda + \delta_4  \le \lambda + \chi + \tau. \label{eq:uni2}
 \end{equation}
  Using Lemma \ref{lem:intersect} we  deduce from  (\ref{eq:red1}),  (\ref{eq:uni1})   and   (\ref{eq:uni1})  the first  assertion of  Lemma
  \ref{lem:uniform} immediately.  The second assertion follows  from  the first  assertion and  Proposition \ref{prop:rank}.
 This completes   the  proof  of  Lemma \ref{lem:uniform}. 
\end{proof}

%\underline{Step 5}.   Thanks to Proposition \ref{prop:exists} we assume that  $J$ is regular,  $\Ht_\eps$ is $J$-regular.  
%In this  step we extend the chain  groups $CFN_*(\Ht)$  and  $CFN_*(\Ht_\eps)$  by   extending   scalars to  new chain  groups  over   the same field.

%\begin{lemma}\label{lem:ext1}
%\end{lemma}

\underline{Step 5}   In this  step we   shrink  the   chosen interval $[-\delta_4, \delta_4]$  to a  smaller  sub-interval $[-\delta, \delta]$   and  shrink   the  good neighborhood   $U(\Ff')$  to a  ``better''   neighborhood  $U_{c'} (\Ff')$,  where   $\Ff'$  is the subfamily of $\Ff$   with parameter  $\chi \in [-\delta, \delta]$.  This is necessary   for the   proof  of  different  energy estimates, which  we use  in establishing  a chain map  between the  chain complexes     arising from the map in (\ref{eq:chainh2}). %validity  of Lemma \ref{lem:estn}, which is  the main   ingredient of the  proof of Lemma \ref{lem:finite1}, which .

First, $\delta= \delta (\Ht)$ must  satisfy  the following  two conditions.
\begin{equation}
|| \delta  \cdot \theta||_{ C^0 (S^1 \times  M^{2n})}  < \frac{\delta_1(\Ht)}{12}, \label{eq:d16}
\end{equation}
\begin{equation}
||d(h  ^0 _\chi - h ^0 _{\chi'})||_{C^0 (S^1 \times  M^{2n})}  < \frac{\delta_1(\Ht)}{12},  \text{ for any } \chi , \chi ' \in [-\delta, \delta].\label{eq:d16c}
\end{equation}
Further, we     define   %choose  for  all  $\chi  \in [-\delta_3, \delta_3]$  we have
\begin{equation}
U_{c'} (\Ht_\chi) : = U_c (\Ht_\chi) \cap\{ \Ht_\chi + \pi^* (h_\chi)|\,  ||dh_\chi|| _{ C^0(S^1 \times  M^{2n})}< \frac{\delta_1(\Ht)}{12}\} .\label{eq:d16b}
\end{equation}

%Now we set
%$$\delta= \delta (\Ht) : = \delta _3.$$

\begin{lemma}\label{lem:estn}  Let  $\chi \in [-\delta, \delta]$,  $\Ht_{\mu(0)}\in  U_{c'} (\Ht_0)$ and   $ \Ht_{\mu(\chi)}\in U_{c'} (\Ht_\chi)$.
Then   there  exists $\tilde h _{\mu(\chi)} \in C^\infty_{(*)} (S^1 \times  \Mt^{2n})$ such that
$$\Ht_{\mu (\chi)} = \Ht_{\mu(0)} + \tilde h _{\mu (\chi)},$$ 
$$||   d \tilde h_{\mu(\chi)}||_{ C^0(S^1 \times   M^{2n})}  < \frac{\delta_1 (\Ht)}{3}. $$
\end{lemma}

In Lemma \ref{lem:estn},  under   the norm $||   d \tilde h_{\mu(\chi)}||_{ C^0(S^1 \times   M^{2n})}$   we mean  the norm    $||\theta'||_{C^0(S^1 \times   M^{2n})}$  where $d \tilde h_{\mu(\chi)} = \pi^*(\theta')$.   The 1-form $\theta'$     exists  uniquely, since
$\tilde h _{\mu(\chi)} \in C^\infty_{(*)} (S^1 \times  \Mt^{2n})$.

\begin{proof} By  (\ref{eq:d16b}), (\ref{eq:htilde}), (\ref{eq:Htilde}), we have
$$\Ht_{\mu(\chi)} -\Ht_{\mu(0)}  = (\Ht_\chi  +\pi ^* (h_\chi)) - (\Ht_0  + \pi ^* (h_0))  ) $$
$$ =  \tilde h_\chi ^ 0  - \tilde h _0 ^0   +  \pi ^* (h_\chi) -\pi ^* (h_0)$$
$$ = \chi \cdot h ^ \theta  + (\pi ^* (h_\chi^0) -\pi ^* (h_0^0)) +  \pi ^* (h_\chi) -\pi ^* (h_0).$$
Now   using (\ref{eq:d16}, (\ref{eq:d16b})  and  (\ref{eq:d16c}) we obtain
$$||   d h_{\mu(\chi)}||_{ C^0(S^1 \times   M^{2n})}  < 4 \cdot ( \frac{\delta_1 (\Ht)}{12}) =\frac{\delta_1 (\Ht)}{3},$$
what is required   to  prove.
\end{proof}

  To   simplify    notations  we  shall    re-denote in  the remainder  of this subsection $c'$ as  $c$   and we let
$$CFN_* ^{red} (\Ht_{\mu(\chi)}),  \F) : =\Lambda^ \F _{\theta(\lambda - \delta, \lambda + \delta), \om} \otimes_{\F[\Gamma^0]} CFN_* ^0 (\Ht_{\mu(\chi)}),  \F)$$
for any $\chi \in [-\delta, \delta]$.

 Now  we are ready  to   define a  linear mapping   between  
  $CFN_* ^{red}(\Ht_{\mu(0)}, \F)$ and  $CFN_*^{red}(\Ht_{\mu(\chi)}, \F)$  for $\chi \in [-\delta, \delta]$.   
Let $\phi (s)$ be a monotone increasing smooth function on $[-R, R]$ which vanishes
near $-R$ and equals $1$ near $R$.  Taking into account Lemma \ref{lem:estn}, we set
with 
$$(J_s,\Ht_{s,t})=(J,(\Ht_{\mu(0)}) _t)  \text{~for~} s<-R,$$
$$(J_s,\Ht_{s,t})=(J,+  (\Ht_{\mu(\chi)})_t)  \text{~for~} s>R, $$
$$\Ht_{s,t}= (\Ht_{\mu(0)})_t + \phi (s) \cdot (\tilde h_{\mu (\chi)})_t.$$

We  consider the space  $\Mm([\xt, \wt^-], [\yt, \wt^+],\Ht_{s,t}, J_s)$  of the solution  $\ut: \R \times S^1 \to \Mt^{2n}$ of the following Floer  chain map equation  (cf: (\ref{eq:conl}),
\ref{eq:conb}), \ref{eq:conh})
\begin{equation} 
\frac{\partial u}{\partial s}+ J_s(u)(\frac{\partial u}{\partial t}-
X_{\Ht_{s,t}})=0, \label{eq:chainh}
\end{equation}
with the following boundary conditions
\begin{equation}
\lim_{s\to -\infty} = \xt \in \Pp (\Ff),\label{eq:chainb+}
\end{equation}
\begin{equation}
\lim_{s\to \infty} = \yt \in \Pp (\Ff),\label{eq:chainb-}
\end{equation}
\begin{equation}
[\xt, \wt^- \# \ut] = [\yt, \wt^+] \label{eq:chainhh}
\end{equation}

\begin{lemma}\label{lem:finite1}  $\Mm([\xt, \wt^-], [\yt, \wt^+], \Ht_{s,t}, J_s)$ is  a finite  set  if
$\mu([\xt, \wt^-]) =  \mu([\yt, \wt^+])$.
\end{lemma}

\begin{proof}  We define the  energy of   a solution  $\ut  \in  \Mm([\xt, \wt^-], [\yt, \wt^+], \Ht_{s,t}, J_s)$  as follows
$$E(u)  = \int_\infty ^\infty \int_0^1 |\frac{\p u}{\p s}| ^2 \, dt ds.$$
The following estimate has been obtained  in \cite{LO1995}. 

\begin{lemma}\label{lem:energy1} (\cite[Lemma 5.4]{LO1995}) %Let $\epsilon$ be a real number such that 
Assume that $\vert d\tilde h_{\mu(\chi)} \vert_{C^0(S^1\times M^{2n})} <\delta_1/3$.  
For a solution $\tilde u$ of (\ref{eq:chainh}) satisfying (\ref{eq:chainb+}, \ref{eq:chainb-}) and  (\ref{eq:chainhh})
we have
%the homotopy condition 
%$[\tilde{z}^-,\tilde w^- \# \tilde u] = [\tilde{z}^+,\tilde w^+]$, we have  
$$E(\tilde u) \leq 3(\Aa_{\Ht_{\mu(0)}}([\xt,\tilde w^+])-\Aa_{\Ht_{\mu(\chi)}}([\yt,\tilde w^-])).$$
%where $\Ht_t$ is a Hamiltonian function on $\Mt$ such that 
%$d\Ht_t=\pi^*\theta_t$.  
\end{lemma}
Recall that our choice of $\tilde h_{\mu(\chi)}$  satisfies the estimate in Lemma \ref{lem:estn}  and hence  the condition of Lemma \ref{lem:energy1}  is fulfilled.
 Hence  the energy  of   a solution  $\ut  \in  \Mm([\xt, \wt^-], [\yt, \wt^+], \Ht_{s,t}, J_s)$  is uniformly bounded.
The  weak  compactness theorem   yields    Lemma \ref{lem:finite1}  immediately.
\end{proof}

Now we  define  a   map $\psi:  \Pp(\Ht)  \to  CFN_*^{red} (\Ht_{\mu (\chi)}, \F)$ as  follows
\begin{equation}
\psi (\xt): = \sum_{\mu(\yt) = \mu (\xt)} m (\xt, \yt) \cdot \yt \label{eq:chainh2}
\end{equation}
where $m (\xt, \yt)$ denotes the  algebraic  cardinality  of $\Mm(\xt, \yt, \Ht_{s,t},  J_s)$.

Lemma \ref{lem:finite1}   and the existence  of  a coherent  orientation on the moduli  space of    the solutions of the Floer chain map equation
imply that  the     number $m(\xt, \yt)$ is well  defined,  but there  are possibly  infinitely many   terms in the   RHS of (\ref{eq:chainh2}).
So we need the following

\begin{lemma}\label{lem:5.5} Given $\xt \in \Pp (\Ht)$  and a  number $ C \in \R$   there  exists only finitely many  $\yt$  in  the  RHS of (\ref{eq:chainh2})  such that   $\Aa_{\Ht_{\tau}} (\yt) > C$
for any $\tau \in [-\delta, \delta]$.
\end{lemma}

\begin{proof} This Lemma  is an  analogue of   the relation (\ref{eq:finite}).  As   the  proof of  (\ref{eq:finite})  is based on
the uniform   energy estimate in Lemma \ref{lem:onolem34}    our  proof is based  on the  following estimate  in \cite[Lemma 3.5]{Ono2005}  for  a solution $\ut   \in  \Mm([\xt, \wt^-], [\yt, \wt^+], \Ht_{s,t}, J_s)$, which states  that
\begin{equation}
\Aa_{\Ht} ([\xt, \wt^-]) - \Aa_{\Ht}([\yt, \wt^+]) > \frac{\delta_1(\Ff)}{6} \rho (\xt, \yt).\label{eq:ono35}
\end{equation}
Repeating the argument in the  proof  of  (\ref{eq:finite}), using (\ref{eq:ono35})  instead  of Lemma \ref{lem:onolem34},   we obtain   immediately Lemma \ref{lem:5.5}.
This   completes  the  proof  of Lemma \ref{lem:5.5}.
\end{proof}

\

\underline{Step  6}.  In this  step   we complete  the  proof of Proposition \ref{prop:bettismall}.
By Lemmas \ref{lem:finite1}, \ref{lem:5.5},  the  map $\psi$ in (\ref{eq:chainh2}) extends   uniquely  as  a chain map, which we also denote by $\psi$:
$$\psi: CFN_*^{red} (\Ht_{\mu(0)}, \F)  \to CFN_*^{red} (\Ht_{\mu (\chi)}, \F). $$
We need to show that   $\psi$  is a chain homotopy equivalence.  %We observe that  the  chain map $\psi':CFN_*^{red}(\Ht_{\mu(\chi)})  \to CFN_* ^{red}(\Ht_{\mu(0)})$  is well  defined, 
 % thanks to Lemmas  \ref{lem:finite1}, \ref{lem:5.5}.
The  proof of    the chain homotopy equivalence      is  proceeded  using   standard arguments  as in   the  proof of  \cite[Theorem 5.3]{LO1995},  see also \cite[Theorem 4.6]{Ono2005}. 
Combining with Lemma \ref{lem:uniform}, this completes  the proof  of Proposition \ref{prop:bettismall}
\end{proof}
As  we have remarked,   Proposition \ref{prop:bettismall} yields   Theorem \ref{thm:small}.
\end{proof}

%%%%%%%%%%%%%%%%%%%%%%%%%%%%%%%%%%%%%%%%%%%%%%%%%%%%%%%%%%%%%%%%%%%%%%%%%%%%%%%%%%%%%%%%%%%%%%%%%%%%%%%%%%%%%%%%%%%%%%%%%%%%%%%%%%%%%%%%%%%%%%%%%%%%%%%%%%%%%%%%%%%%%%%%%%%%%%%%%%%%
%%%%%%%%%%%%%%%%%%%%%%%%%%%%%%%%%%%%%%%%%%%%%%%%%%%%%%%%%%%%%%%%%%%%%%%%%%%%%%%%%%%%%%%%%%%%%%%%%%%%%%%%%%%%%%%%%%%%%%%%%%%%%%%%%%%%%%%%%%%%%%%%%%%%%%%%%%%%%%%%%%%%%%%%%%%%%%%55555
%%%%%%%%%%%%%%%%%%%%%%%%%%%%%%%%%%%%%%%%%%%%%%%%%%%%%%%%%%%%%%%%%%%%%%%%%%%%%%%%%%%%%%%%%%%%%%%%%%%%%%%%%%%%%%%%%%%%%%%%%%%%%%%%%%%%%%%%%%%%%%%%%%%%%%%%%%%%%%%%%%%%%%%%%%%%%%%%%%%%

\subsection{Computing the Betti numbers of   $HFN_*(\pi^*(H),\F)$}\label{subs:comp} Let $\F$ be a field.
 In this  subsection  we      compare the  sum of the Betti numbers of the Floer-Novikov homology $HFN_*(\pi^*(H),\F)$ with   the  sum  of the Betti  numbers  of the Novikov homology  $HN_*(M, [\theta], \F)$. 
  It is known that the  latter ones can be computed 
via  the refined Morse  complex $CM_* (\pi ^* (f), \p ^{Morse}, \F)$  with coefficients in $\F$ of a  lifted  Morse  function $f\in C^\infty (M^{2n})$,
see  e.g. \cite{Farber2004, Pajitnov2006}. 
Recall that $\Lambda  ^\F_{0, \om}$ is the  completion of  $\F[\Gamma ^0]$  w.r.t.   the weight homomorphism $0 \oplus - I_\om$, see  (\ref{eq:novikov}).
By Proposition \ref{prop:direct2}, $\Lambda^\F_{0,\om}  = \F[\Gamma_1]((-I_\om(\Gamma_2^0)))$. 
Thus  $\Lambda_{0, \om}^\F$  is   the underlying   ring of the   Floer-Novikov chain complex   $CFN_*(\pi^*(H), \F)$, which is also called {\it  a  refined   Floer  chain complex}  by Ono-Pajitnov \cite{OP2014}, see also \cite{LO2015} for   further discussion. 
%
%Our  goal is to  prove the following  Theorem. %To finish  the   computation of the Betti numbers of the   Novikov  homology 

\begin{theorem}\label{thm:comp2}  Assume that $J$ is a regular compatible almost complex structure   and  $H\in C^\infty (S^1 \times M^{2n})$  is nondegenerate  and $J$-regular.
There  is  a chain isomorphism  between  the
Floer-Novikov chain complex  $CFN_*( \pi^*(H), J, \F)$  and the $\Z_{2N}$-graded extended Novikov  chain complex $(\Lambda _{0,\om}^\F  \otimes _{\F[\Gamma_1]} \tau_{2N}(CM_{*+n}(\tilde  f)), Id \otimes\p ^{Morse})$,
where $\tau_{2N}$  denotes  the natural projection  of $\Z$-grading  to $\Z_{2N}$-grading.  Consequently\\
   $\sum b_i (HFN_*(\pi^*(H), \F)  = \sum b_i (HN_*(M; [\theta], \F)).$
\end{theorem}

Theorem \ref{thm:comp2}  is a partial  case  of Theorem 2.8  in \cite{OP2014}  which considers    refined  Floer complexes  over $\Z$, see also
Theorem  2.9 in \cite{OP2014}   concerning  general compact symplectic manifolds.    Ono and Pajitnov omit  the   proof of    \cite[Theorems  2.8, 2.9]{OP2014}, which  can be  done in the same way as  in the   Floer  homology  theory.   For the  reader  convenience   we shall outline   the  proof
of  Theorem  \ref{thm:comp2}, using the Piunikhin-Salamon-Schwarz scheme  for comparing  Floer  chain complexes with  Morse chain  complexes  \cite{PSS1996}.  We use the analytical   framework  developed by Oh-Zhu in \cite{OZ2011}  and by Lu in \cite{Lu2004}, see also  \cite{OZ2012} for another  analytical  approach.
%To  prove  Theorem  \ref{thm:comp} we   extend  the  Pinhiukhin-Salamon-Schwarz  construction of chain homotopy equivalence  between   the  chain complexes in Theorem \ref{thm:comp}  for the case $\theta = 0$ \cite{PSS1996}, see also \cite[p.460-463]{McDS2004} for an explanation.   Here  we shall  follow the   scheme  proposed   by Oh-Zhu in \cite{OZ2011},  slightly  simplifying   some analytical settings.  In fact  there  is no difficulty, since  we  have the  same energy  estimate... %, even  for the  case $\theta = 0$.

{\it Outline of the proof of Theorem \ref{thm:comp2}}.
Let  $H \in  C^\infty (S^1 \times M^{2n})$  be   a  nondegenerate  Hamiltonian and $J$-regular for   some    regular  compatible almost complex structure  $J$.  %Then %$CFN_*(\p^*(H), J)$  be  the  Floer-Novikov chain  complex on the covering  $\tilde M$.   %and  $(\Lambda _{\theta, \om} ^R \otimes _{R[\Gamma_1]} CM_* (\tilde  f), Id \otimes\p ^{Morse})$  the reduced quantum Novikov complex. 
We  shall define   a chain  map  
$$\Phi: (\Lambda_{0, \om}^\F\otimes _{\F[\Gamma_1]} \tau_{2N}(CM_* (\pi^*( f))), Id \otimes\p ^{Morse}) \to  CFN_*(\pi^*(H), J, \F) $$
and   a chain  map
$$ \Psi: CFN_*(\pi^*(H), J)  \to   (\Lambda_{0, \om}^\F \otimes _{\F[\Gamma_1]} \tau_{2N}(CM_* (\pi^*(f))), Id \otimes\p ^{Morse})$$
by  defining  the value of   $\Phi$ and $\Psi$ on  the  generators  of each involved module, and then  extending    $\Phi$ and  $\Psi$ linearly over the   ring  $\Lambda  _{0,\om}^\F$.

Recall that  the generators  of $\Lambda_{0, \om}^\F\otimes _{\F[\Gamma_1]} CM_* (\pi^*(f))$ are critical points   of  $\tilde f$, which we shall denote by $\tilde p$. Note  that   $\tilde p : = \pi^{-1} (p)$, where $p$  are  critical points  of $f$.
The generators   of   $CFN_*(\pi^*(H), J)$
are  the critical points  $[\xt,  \wt] \in \Ppt (\pi^*(H))$ of the action functional $\Aa _{\pi^*(H)}$.  Similarly we have
$\Ppt (\pi^*(H))  = \pi^{-1}(\Ppt (H))$.
 
We define  $\Phi$ by the following formula  %we  define the  entries   of   the matrix   associated  to  $\Phi$  as follows.
\begin{equation}
\Phi (\tilde p) : = \sum _{[\tilde  x, \tilde w] \in \Ppt(\pi^{-1} (H))} m (\tilde p,   [\tilde x,   \tilde w]; A ) [\tilde  x, \tilde w] e ^ {-A}, \label{eq:Phi} 
\end{equation}
where $A \in \Gamma_0$. 
We   define $\Psi$ as follows
\begin{equation}
\Psi ([\tilde x, \tilde w]): = \sum_{\tilde p \in Crit (\tilde  f)} n ([\tilde  x, \tilde w], [\tilde  p, A] )  \tilde  p\cdot e^ A, \label{eq:Psi}
\end{equation}
where  $A  \in \Gamma_0$.    The coefficient $m (\tilde p,   [\tilde x,   \tilde w]; A )$  (resp. $n ([\tilde  x, \tilde w], [\tilde  p, A])$)    will be defined
as  the algebraic cardinal of the moduli space  $\Mm (\tilde  p,  [\tilde x,   \tilde w]; A )$  (resp. $\Mm ([\tilde  x, \tilde w], [\tilde  p, A])$) of spiked disks  that consist of (perturbed) holomorphic  disks   equivalent to $ [A\# -\tilde w]$   and bounding  $\tilde  x$ and  of
 a spike   which is a gradient flow line from  $\tilde p$  to a  holomorphic disk in consideration.

In the first  step  we shall  describe     the moduli spaces of  involved spiked disks. Then  we shall  show that    $\Phi$ and $\Psi$ are well-defined  and
 $\Gamma^0$-invariant  and therefore   extensible  over $\Lambda^\F_{0, \om}$.  Finally  we shall  show that   $\Psi $ and $\Phi$ are  chain  homotopy equivalences.

\

{\it Description of the moduli  space $\Mm (\tilde  p,  [\tilde x,   \ut]; A )$}.   Let %\\
%$\bullet$ $\Sigma$ denote  the Riemannian sphere,\\
%$\bullet$ $\dot \Sigma$ - the punctured  sphere, \\
%$\bullet$ 
$\dot \Sigma_+$   (resp.  $\dot\Sigma _-$) denote the  the punctured  sphere %complex plan $\C$ % with polar  coordinate $(\tau, t)$  %
with one    marked point  $o_+$  (resp.  $o_-$)  and one positive   puncture $e_+$ (resp. $e_-$).
%We choose  analytical charts  at $o_+$  and   at $e_+$   on some  neighborhood $O_+$ and $E_+$ respectively  so that conformally 
%$$O_+ \setminus o_+  \cong (-\infty, 0]   \times S^1, $$
%$$E_+ \setminus e_+  \cong  [0, + \infty)  \times S^1.$$  
We fix  an identification   
\begin{equation}
\dot\Sigma_\pm \setminus  \{  o _\pm\}  \cong\R \times  S^1\label{eq:can}
\end{equation}
 and denote  by  $(\tau, t)$  the corresponding  coordinates   so that  
$\{ \mp \infty\} \times S^1$  and  $\{ \pm \infty\}  \times S^1 $ correspond   to $o_\pm$ and   $e_\pm$  respectively.

%Fix a     Riemannian metric  $h_\Sigma$ 
%Now we are going to define  the notion  of  a perturbed  Cauchy Riemann  equation for  a map $u_\pm :\dot \Sigma_\pm \to M^{2n}$, similar to that one defined in (\ref{eq:conl}).
 %First we define a homotopy  $K_\pm: \R \times  S^1 \times  \Mt^{2n} \to \R$  between  the  Hamiltonian $\pi^*(H)$  and   the zero function as in  Step 6 of the  previous subsection. %In this section  we  pose  a  normalization condition 
 %on $\Ht$  as follows.    Let  us  fix  a point $x_0 \in \Mt$. Then  we require  that
 %\begin{equation}
 %\Ht (t, x_0)  = 0 \, \forall  t \in S^1.\label{eq:norm1}
 %\end{equation}
  Set     %First  let us denote by  $h_0: \tilde  M \to \R$   the function
%such that  $\pi^ *(\theta)  =  dh_0$.
$$ K_\pm(s, t, x) : = \beta _\pm (s) \cdot \pi^*(H) (t, x), $$
where  $\beta_\pm : \R \to [0,1]$  is a smooth cut-off function  given by
\begin{eqnarray*}
\beta_+ (s)  & =  & 0  \text{ for } s \le 0,\\
  &    =   & 1 \text{ for } s \ge 1,\\
	\beta_- (s) & : =  &  \beta_+ (-s)
	%\beta_- (s) & =  & 0 \text{ for } s \ge 0,\\
	% &    =   & 1 \text{ for } s \le 1\\
\end{eqnarray*}
such that
$$ 0< s _+  ' (\tau) < 2 \text{ for } s \in (0,1).$$
%$$ 0> \beta _-  ' (s) >- 2 \text{ for } s \in (0,1).$$

%Clearly $K$ extends    to a smooth map, which  we also denote by $K$:
%$$ K: (\dot\Sigma_+) \times   \Mt  \to \R.$$

%For a map  $u_+ : \dot \Sigma_+  \to 

%We fix  a metric  $h$ on $\dot\Sigma_+$ such that  w.r.t.    the given  cylindrical  coordinates on $\dot \Sigma_+ \setminus   e_+$ we have  $h = d \tau ^2  + dt ^2$  for  $\tau  \ge 0$.

% Let $j$  be a complex  structure on $\dot \Sigma_+$  such that in the cylindrical coordinates $(s, t)$   at $s > 0$ we have
%$j (\p s) = \p  _t$.

 For    a  regular   compatible almost complex structure $J$ on $M^{2n}$ we define  a perturbed  Cauchy-Riemann operator
$\bar \p _{(K_\pm, J)}$   acting on the  space of smooth  mappings $u_\pm : \dot \Sigma_\pm \setminus  \{  o _\pm\}  \to \Mt$ as follows.
 % In cylinder  coordinates $(s, t)$  we     set 
$$\bar \p _{(K_\pm, J)} u_\pm : = \p_s u + J (u)  (\p_t  u  - \beta_\pm(s)  X_{\pi^* (H)}).$$
 Assume  that $\bar \p _{(K_\pm, J)}u = 0$   and   the  energy
$$E ( u _\pm): = \int_\infty ^ \infty |\p _s  u _\pm |^2 _{ g_J}  \, ds dt < + \infty.$$
Then  $u_\pm$  extends to  $\dot\Sigma_\pm$, since nearby   the marked  point $o_\pm$ the  perturbation term  $K_\pm$  vanishes.
%For any $x \in \dot \Sigma_+$  we set
%\begin{eqnarray}
%\bar \p _{(K, J)} u: &  &  T_x \dot \Sigma _+\to T_{u_+ (x)}M^{2n}\nonumber \\
 %& v  &\mapsto   du_+  (v)  + J\circ  du (j\circ v)   + \nabla K (x, u_+(x))).\label{eq: pk+}   
%

%In the next step we want to define an invariant expression  for  the  energy of  a map $u_+$ as  in \cite{OZ2011}.\footnote{there is a problem in  defining the $(0,1)$-component of  a  1-form  since  $ham  (M, \om)$ is   not $J$-invariant!}
%We set
%$$C^\infty (\Mt, x_0, \Gamma_1) : = \{  f \in C^\infty (\Mt)|\,  f(x_0) = 0, \,    f (g \cdot  x_0) = f (x_0) + I_\theta (g) \, \forall g \in  \Gamma_0\}.$$

%Then $C^\infty (\Mt, x_0, \Gamma_1)$ is a linear  space that contains $\Ht _t$ defined in (\ref{eq:norm1}).

%\begin{lemma}\label{lem:exi}
%\end{lemma}
%for $(\tau, t)  \in \dot \Sigma \setminus o_+$.  Since $K = 0$  for $\tau \le 0$,   the  operator
%$\bar \p _{(K_+, J_+)}u_+$ is well-defined   on  $\dot \Sigma_+ $.

%For $J_+ \in \Jj (M^{2n}, \om)$  and $[\tilde x, \tilde w] \in \Ppt (\Ht)$  we define
We set
\begin{eqnarray*}
\Mm_\pm (\tilde  p,  [\tilde x,   \wt]; A )  & = & \{ (\chi_\pm, u_\pm)| u _\pm : \dot\Sigma_\pm \to \Mt, \, [u_\pm\#  \wt]  = A,\\
                         &   & u_\pm(\pm\infty, t) = \tilde x (t), \bar\p_{(K_\pm, J)}  u_\pm  = 0, \\
	&  & 	\dot  \chi _\pm  = \nabla (\tilde f (\chi_\pm)), \, \chi_\pm (\mp \infty) = \tilde p, \, \chi_\pm (0) = u_\pm (o_\pm)\}.
\end{eqnarray*}
In the  above  expression   we require that
$$ m(\xt, \wt]) = n - (\mu_{ ( \pm  f)}(p) $$
so that
$$\dim \Mm_\pm (\tilde  p,  [\xt,   \wt]; A] )  = 0.$$
Then we  set
\begin{equation}
m (\tilde p, [\xt, \wt]; A): = \#\Mm_+ (\tilde  p,  [\xt,   \wt]; A_+ ). \label{eq:Phi}) 
\end{equation}

Similarly we set
%\begin{eqnarray*}
%\Mm (  [\tilde x_-,   \tilde w _-], \tilde  p; A_-] )  & = & \{ (\chi_-, u_-)| u _+ : \dot\Sigma_- \to \Mt, \, [u_-\#  w_-]  = A_-,\\
%                         &   & u(-\infty, t) = \tilde x_- (t), \bar\p_{(K_-, J_-)}  u_-  = 0, \\
%	&  & 	\dot  \chi _-  = \nabla (\tilde f (\chi_-)), \, \chi_+-(\infty) = \tilde p, \, \chi_-  (0) = u_- (o_-)\}.
%\end{eqnarray*}
%In the  above  expression   we require that
%$$ \mu(\zt_+, \wt_+]) = n - (\mu(p) - 2c_1 (A_+))$$
%so that
%$$\dim \Mm (\tilde  p,  [\tilde x_+,   \tilde w]; A_+] )  = 0.$$
%Then we  set
\begin{equation}
n ([\xt, \wt],[\tilde p,  A]): = \#\Mm_- (\tilde  p,  [\tilde x,   \tilde w]; A ). \label{eq:Psi}) 
\end{equation}

\begin{theorem}\label{thm:pssn} The maps   $\Phi$ and  $\Psi$ defined   in  (\ref{eq:Phi}) and (\ref{eq:Psi})     are well-defined   and they are $\Gamma_1$-equivariant.
Furthermore, the compositions $\Psi\circ \Phi$ and $\Phi\circ   \Psi$    are    chain homotopy equivalent to the  identity.
\end{theorem}

{\it Outline of the proof}. 
The first   assertion of   Theorem \ref{thm:pssn}    follows   from the   compactness   and the coherent  orientability  of  the moduli  spaces  $\Mm_\pm (\tilde  p,  [\xt,   \wt]; A)$.  The compactness is  proved by
   using    an  upper estimation for  the energy of   the  perturbed holomorphic  disks $u_\pm$  involved in the   moduli spaces  in consideration.  Recall that
%For $u_\pm : \dot \Sigma_\pm  \to \Mt$  we define  its  energy as  follows.  Using (\ref{eq:can})  we write
\begin{equation}
E(u_\pm) : = E(u_\pm |_{\dot \Sigma_\pm  \setminus  \{ o_\pm \}} ) = \int_{-\infty} ^\infty \int_0 ^ 1 | \p _\tau  u| _{g_J} ^2   dt  d\tau. \label{eq:energy}
\end{equation}

\begin{lemma}\label{lem:energy2}   Assume that  $u_+ \in \Mm_+ (\tilde  p,  [\xt,   \wt]; A)$. Then
$$E(u)  \le  \om (A) - \Aa_H ([x, w]) + 2 \max_{ (t, x) \in S^1 \times M^{2n}}   H (t, x).$$
\end{lemma}
\begin{proof} We write
$$ E(u_+)  = \int_{-\infty} ^\infty \int_0 ^ 1 | \p _\tau  u| _{g_J} ^2   dt  d\tau  = \int_{-\infty} ^{\infty} \int _0 ^1 \la  \p_\tau u,  J \p_t u + \beta(\tau)\nabla  H( t, u) \ra dt d\tau $$
$$ =  \om (A)   -\om (\tilde w) - \int_{-\infty} ^\infty\int _0^1 \beta(\tau) {d\over ds} H( t, u)\, dt d\tau $$
$$ = \om (A) -\Aa_{\Ht} ([\tilde x, \tilde w]) +\int_0 ^1  \beta ' (\tau) \int_0 ^1  H(t,  u)\, dt d\tau ,$$
since  $\beta '  (s) = 0$ for  $s \in \R \setminus [0,1]$.    Taking into account $0 \le  \beta _+ (s) ' \le  2$, we obtain Lemma  \ref{lem:energy2} immediately.
\end{proof}

Lemma   \ref{lem:energy2}  provides  the    weak  compactness   of   the  moduli space $\Mm_+ (\tilde  p,  [\xt,   \wt]; A)$.
In the same  way we   obtain the    weak  compactness   of   the  moduli space $\Mm_- (\tilde  p,  [\xt,   \wt]; A)$.
The regularity of $J$ and  the $J$-regularity of $H$  yield  the compactness  of the moduli  spaces  in considerations.  The  coherent   orientation of   the moduli   space is defined
  as in the  \cite{FH1993}.
  
 \
  
The  second  assertion of  Theorem  \ref{thm:pssn}    follows  from the fact, that on the covering  space $\Mt$  all    objects  under consideration are $\Gamma_1$-invariant.
Finally     the  last assertion   of Theorem \ref{thm:pssn}  is  proved in the  same way  as   in the Floer homology case, so we omit  the   proof.

We obtain immediately from Theorems \ref{thm:comp}, \ref{thm:comp2}  the following.

\begin{corollary}\label{cor:betti}  For any  nondegenerate $J$-regular  Hamiltonian function $\Ht \in  C^\infty _{(*)} (S^1 \times \Mt)$ we have
$$\sum_i b_i (HFN_*(\Ht, \F)) = \sum_i  b_i (HN_*(M, [\theta]), \F).$$
\end{corollary}

%%%%%%%%%%%%%%%%%%%%%%%%%%%%%%%%%%%%%%%%%%%%%%%%%%%%%%%%%%%%%%%%%%%%%%%%%%%%%%%%%%%%%%%%%%%%%%%%%%%%%%%%%%%%%%%%%%%%%%%%%%%%%%%
%%%%%%%%%%%%%%%%%%%%%%%%%%%%%%%%%%%%%%%%%%%%%%%%%%%%%%%%%%%%%%%%%%%%%%%%%%%%%%%%%%%%%%%%%%%%%%%%%%%%%%%%%%%%%%%%%%%%%%%%%%%%%%%%%%
%%%%%%%%%%%%%%%%%%%%%%%%%%%%%%%%%%%%%%%%%%%%%%%%%%%%%%%%%%%%%%%%%%%%%%%%%%%%%%%%%%%%%%%%%%%%%%%%%%%%%%%%%%%%%%%%%%%%%%%%%%%%%%%

\section{Proof of Theorem  \ref{thm:main}}\label{sec:main}

%\subsection{The weakly monotone case}
The second  assertion   of Theorem  \ref{thm:main}  follows  immediately from  Corollary  \ref{cor:betti}  and Lemmas \ref{lem:field}.(ii), \ref{lem:chain}.

To prove the first assertion  of  Theorem \ref{thm:main}    we  consider the Floer-Novikov  chain complexes  and their  homology with coefficients in $\Q$  as  in  \cite{Ono2005,  FO1999}. We also  refer  the reader to \cite{FOOO2015} for the latest account  of the  Fukaya-Oh-Ohta-Ono theory of   Kuranishi structures  and their applications.
Let  us       rapidly recall    the construction  of the   Floer-Novikov chain complexes  on general compact symplectic manifolds.
 First we compactify  the quotient spaces $\Mm([\xt, w^-], [\yt,  w^+])/\R$. % of connecting orbits $u$
%where  the group $\R$  acts  on  connecting orbit $u : R \times S^1 \to \Mt^{2n}$ by   translation in the  variable $s\in \R$.  
The compactified space $\overline{\Mm} ([\xt, w^-], [\yt, w^+])$
is obtained from  the quotient  space  $\Mm([\xt, w^-], [\yt,  w^+])/\R$   by   adding stable connecting orbits as in \cite[Definitions 19.9,  19.10,  p.1018-1019]{FO1999}.
Here  we do not require  the   regularity of    a  compatible almost complex structure  $J$   and  the $J$-regularity  of a Hamiltonian  function $\Ht \in C^\infty _{(*)}(S^1 \times \Mt^{2n})$.
There  exists   a natural Kuranishi structure  on $\overline{\Mm}([\xt, w^-], [\yt,  w^+])$. Using  abstract  multi-valued  perturbation technique, we can define  the  ``algebraic cardinality" $\la [\xt, w^-], [\yt, w^+]\ra  \in \Q$ when  $\mu([\xt, w^-] - \mu([\yt,  w^+]) = 1$  and   set
$$\p _{J, \Ht}  [\xt, w^-]  : = \sum \la [\xt, w^-], [\yt, w^+]\ra [\yt, w^+], $$
where $[\yt, w^+]$ runs  over the set  of critical points  of $\Aa_{\Ht}$  such that  $\mu([\xt, w^-] - \mu([\yt,  w^+]) = 1$.
It is known  that $\p _{J, \Ht} ^2 = 0$.  The resulting  homology is called   the Floer-Novikov homology of the   Floer-Novikov chain complex $CFN_*(\Ht, J, \Q)$.   We  also   know that  the  Floer-Novikov homology  is invariant  under Hamiltonian isotopy \cite[Theorem 3.1]{Ono2005}.

To  prove  the invariance   of the  Floer-Novikov homology  $HFN_*(\Ht, \Q)$   where $\Ht\in  C^\infty _{(*)} (S^1 \times  M^{2n})$   we use   a  simplified  argument in the  previous subsection. Namely   it suffices to use    an  admissible family $\Ff'$  defined  in Step 5 (but not  the ``better'' neighborhood  $U_{c'} (\Ff)$  which contains $J$-regular   Hamiltonian functions).   With $\Ff'$ we have all necessary energy estimates
without taking  care on the  $J$-regularity of       perturbed  Hamiltonians $\Ht'$.
This completes the  proof  of  Theorem \ref{thm:main}.

%\end{proof}

\section{Concluding remarks}\label{sec:concl}

1. One   of main   technical  difficulties in    the computation of  the Floer-Novikov  homology  is the variation  of  the  isomorphism
type of the underlying Novikov ring $\Lambda ^R_{\theta, \om}$  under the  variation of   $[\theta]$   inside its   conformal class $\R \cdot [\theta]$.
For example,  when $R $ is a field, by Proposition \ref{prop:compono},  the  Novikov  ring $\Lambda ^R_{\theta, \om}$ is a   field, if and only if 
$\ker \Psi_{\theta, \om} = 0$.   The important idea that  the Floer-Novikov  chain complex  can be defined over  a    proper  sub-ring of  $\Lambda_{\theta, \om} ^R$, which  is constant    for a small variation of $[\theta]$ in  its  conformal class,   has been appeared  first in \cite{Ono2005}.  If there  is a completion  of $R[\Gamma ^0]$   which is an integral domain and  contains   both  different    Novikov rings $\Lambda ^\F_{\theta, \om}$, $\Lambda ^\F_{c\cdot \theta, \om}$ as its sub-rings,   the  proof of the  Main Theorem can be simplified.

2.  In \cite{Seidel1997, Seidel2002}  Seidel  defined  his version of  Floer homology of a symplectomorphism $\phi$ as the Floer homology of the   symplectic   fibration  obtained  from  the   torus mapping of $\phi$.  It is not clear how Seidel's version of Floer-Novikov homology is related to   our construction, especially   how to recognize  the  Calabi invariant   of $\phi$, if $\phi$ is symplectic   isotopic  to     the identity.  
%the  Novikov homology of the   Calabi invariant   of  symplectomorphism under  consideration.

%3.  	The construction of $\eta$ in  Lemma \ref{lem:eta}  can  be applied to the case   the     value of $\theta$ on  each periodic  orbit
%is zero. 
%What does   this mean?  If we know some explicit  non-trivial  %, which implies that   if  such a deformation exists then  the Betti number is locally constant.
%	If the Betti numbers are not  locally constant (e.g  if we  know   some  nontrivial concrete  value) then there   exists  nevertheless    an orbit    for the given
%	Hamiltonian.      Exploit  further?

%3) If the Chern number is equal zero  we have the following  conjecture....

3.  Based  on \cite{FO2001} we conjecture  that  we could  remove  the restriction  of the    field $\Q$  in our   Main Theorem.

4.  In \cite{LO2015}  we  develop    other aspects of    the theory of Floer-Novikov chain complexes to    obtain new   lower bounds  for
the    number  of symplectic  fixed points.  

%4. In view  of ??? we   conjecture  that  the Floer-Novikov   homology can be defined  with  coefficients in any     integral  domain $R$.

\subsection*{Acknowledgement}   I thank Kaoru Ono for   a helpful explanation  during preparation of  this  manuscript.

\end{document}